\documentclass[11pt, a4paper, UKenglish]{amsart}

\usepackage[utf8]{inputenc}
\usepackage[UKenglish]{babel}
\usepackage{amsmath,amscd,amssymb,amsthm, amsfonts}
\usepackage{amssymb}
\usepackage{latexsym,enumerate,graphicx,geometry,rotating,braket}
\usepackage{mathtools}
\usepackage[all]{xy}
\usepackage[style=alphabetic,natbib=true, backend=biber, bibencoding=utf8,maxbibnames=9,sorting=nyt]{biblatex}
\usepackage{thmtools}
\usepackage{thm-restate}
\usepackage[colorlinks=true,hyperindex,breaklinks]{hyperref}
\usepackage{cleveref}
\usepackage{tikz-cd}
\usepackage{stmaryrd}
\usepackage{bbm}
\usepackage{mathrsfs}
\usepackage{subcaption}
\usepackage{multirow}
\usepackage{enumitem}
\usepackage{comment}
\usepackage{mathabx}
\PassOptionsToPackage{linktocpage}{hyperref}

\usepackage{xurl}
\hypersetup{
           breaklinks=true,   
           colorlinks=true,   
           pdfusetitle=true,  
        }

\usepackage{academicons}
\definecolor{orcidlogocol}{HTML}{A6CE39}
\DeclareUnicodeCharacter{0301}{*************************************}
\DeclareUnicodeCharacter{211D}{***************************************}
\theoremstyle{definition}
\newtheorem{theorem}{Theorem}[subsection]
\newtheorem{definition}[theorem]{Definition}

\newtheorem*{notation*}{Notation}

\newtheorem{example}[theorem]{Example}

\newtheorem{remark}[theorem]{Remark}
\newtheorem{corollary}[theorem]{Corollary}
\newtheorem{proposition}[theorem]{Proposition}

\newtheorem{lemma}[theorem]{Lemma}

\newcommand{\ov}{\overline}

\renewcommand{\O}{\mathcal O}
\newcommand{\N}{\mathcal N}
\newcommand{\z}{\mathbf z}

\newcommand{\K}{K}

\newcommand{\T}{\bbT} 
\newcommand{\Tbun}{\sT} 
\newcommand{\mTbun}{\Tbun^{\text{metr}}}
\newcommand{\aTbun}{\widehat{\Tbun}}
\newcommand{\toricbun}{\sX} 
\newcommand{\toricvar}{X} 
\newcommand{\toricmod}{\mathscr{X}} 

\newcommand{\atoricbun}{\widehat{\toricbun}}
\newcommand{\catTbun}{\text{TorBun}} 
\newcommand{\mcatTbun}{\catTbun^{\text{metr}}}
\newcommand{\acatTbun}{\widehat{\catTbun}}
\newcommand{\weightv}{n_v}
\newcommand{\weighti}{n_i}

\let\svthefootnote\thefootnote
\newcommand\freefootnote[1]{%
  \let\thefootnote\relax%
  \footnotetext{#1}%
  \let\thefootnote\svthefootnote%
}

\newcommand{\B}{B}
\newcommand{\mrho}{\rho^{\text{metr}}}
\newcommand{\arho}{\widehat{\rho}}

\newcommand{\ac}{\widehat{c}}

\newcommand{\aP}{\widehat{\calP}}
\newcommand{\aI}{\widehat{I}}
\newcommand{\aD}{\widehat{D}}
\newcommand{\aF}{\widehat{F}}

\newcommand{\IP}{\calP^I}
\newcommand{\catPic}{\mathcal Pic}
\newcommand{\mcatPic}{\mathcal Pic^{\text{metr}}}
\newcommand{\acatPic}{\widehat{\mathcal Pic}}

\newcommand{\Bun}{\text{Bun}}
\newcommand{\intPic}{\widehat{\Pic}^{\text{int}}}

\newcommand{\Q}{\mathbb{Q}}

\newcommand{\Kbar}{\bar{K}}

\newcommand{\A}{\mathbb{A}}

\newcommand{\one}{\mathbbm{1}}

\newcommand{\iso}{\cong}

\newcommand{\ArChow}{\widehat{\textrm{CH}}} 
\newcommand{\ExtChow}{\widecheck{\textrm{CH}}} 

\newcommand{\adeg}{\widehat{\deg}}

\newcommand{\mPic}{\Pic^{\text{metr}}}
\newcommand{\mDiv}{\Div^{\text{metr}}}
\newcommand{\aPic}{\widehat{\Pic}}
\newcommand{\mVect}{\Vect^{\text{metr}}}
\newcommand{\aVect}{\widehat{\Vect}}

\newcommand{\sB}{\mathcal{B}}

\newcommand{\sE}{\mathcal{E}}
\newcommand{\sF}{\mathcal{F}}

\newcommand{\sH}{\mathcal{H}}

\newcommand{\sL}{\mathcal{L}}
\newcommand{\sM}{\mathcal{M}}

\newcommand{\sO}{\mathcal{O}}

\newcommand{\sQ}{\mathcal{Q}}

\newcommand{\sS}{\mathcal{S}}
\newcommand{\sT}{\mathcal{T}}

\newcommand{\sX}{\mathcal{X}}
\newcommand{\sY}{\mathcal{Y}}
\newcommand{\sZ}{\mathcal{Z}}

\newcommand{\bbC}{\mathbb{C}}

\newcommand{\bbP}{\mathbb{P}}
\newcommand{\bbQ}{\mathbb{Q}}
\newcommand{\bbR}{\mathbb{R}}
\newcommand{\bbZ}{\mathbb{Z}}

\newcommand{\bbG}{\mathbb{G}}
\newcommand{\bbT}{\mathbb{T}}

\newcommand{\scrL}{\mathscr{L}}

\newcommand{\calP}{\mathcal{P}}
\newcommand{\calV}{\mathcal{V}}

\newcommand{\into}{\hookrightarrow}

\renewcommand{\div}{\text{div}}

\DeclareMathOperator{\Div}{Div}
\DeclareMathOperator{\Pic}{Pic}

\DeclareMathOperator{\im}{im}

\DeclareMathOperator{\id}{id}
\DeclareMathOperator{\Hom}{Hom}

\DeclareMathOperator{\Spec}{Spec}
\DeclareMathOperator{\Supp}{Supp}

\DeclareMathOperator{\an}{an}

\DeclareMathOperator{\vol}{vol}
\DeclareMathOperator{\ess}{ess}
\DeclareMathOperator{\abs}{abs}

\DeclareMathOperator{\Ext}{Ext}
\DeclareMathOperator{\can}{can}

\DeclareMathOperator{\Fun}{Fun}

\DeclareMathOperator{\Cont}{Cont}
\DeclareMathOperator{\colim}{colim}

\DeclareMathOperator{\QCoh}{QCoh}
\DeclareMathOperator{\Vect}{Vect}

\newcommand{\blank}{{\cdot}}

\DeclarePairedDelimiter\norm{\lVert}{\rVert}%

\let\oldnorm\norm
\def\norm{\@ifstar{\oldnorm}{\oldnorm*}}
\makeatother

\newcommand{\defeq}{\vcentcolon=}

\newcommand\restr[2]{{
  \left.\kern-\nulldelimiterspace 
  #1 
  \vphantom{\big|} 
  \right|_{#2} 
  }}
  
\makeatletter
\let\save@mathaccent\mathaccent
\newcommand*\if@single[3]{%
  \setbox0\hbox{${\mathaccent"0362{#1}}^H$}%
  \setbox2\hbox{${\mathaccent"0362{\kern0pt#1}}^H$}%
  \ifdim\ht0=\ht2 #3\else #2\fi
  }
\newcommand*\rel@kern[1]{\kern#1\dimexpr\macc@kerna}
\newcommand*\wide@bar[2]{\if@single{#1}{\wide@bar@{#1}{#2}{1}}{\wide@bar@{#1}{#2}{2}}}
\newcommand*\wide@bar@[3]{%
  \begingroup
  \def\mathaccent##1##2{%
    \let\mathaccent\save@mathaccent
    \if#32 \let\macc@nucleus\first@char \fi
    \setbox\z@\hbox{$\macc@style{\macc@nucleus}_{}$}%
    \setbox\tw@\hbox{$\macc@style{\macc@nucleus}{}_{}$}%
    \dimen@\wd\tw@
    \advance\dimen@-\wd\z@
    \divide\dimen@ 3
    \@tempdima\wd\tw@
    \advance\@tempdima-\scriptspace
    \divide\@tempdima 10
    \advance\dimen@-\@tempdima
    \ifdim\dimen@>\z@ \dimen@0pt\fi
    \rel@kern{0.6}\kern-\dimen@
    \if#31
      \overline{\rel@kern{-0.6}\kern\dimen@\macc@nucleus\rel@kern{0.4}\kern\dimen@}%
      \advance\dimen@0.4\dimexpr\macc@kerna
      \let\final@kern#2%
      \ifdim\dimen@<\z@ \let\final@kern1\fi
      \if\final@kern1 \kern-\dimen@\fi
    \else
      \overline{\rel@kern{-0.6}\kern\dimen@#1}%
    \fi
  }%
  \macc@depth\@ne
  \let\math@bgroup\@empty \let\math@egroup\macc@set@skewchar
  \mathsurround\z@ \frozen@everymath{\mathgroup\macc@group\relax}%
  \macc@set@skewchar\relax
  \let\mathaccentV\macc@nested@a
  \if#31
    \macc@nested@a\relax111{#1}%
  \else
    \def\gobble@till@marker##1\endmarker{}%
    \futurelet\first@char\gobble@till@marker#1\endmarker
    \ifcat\noexpand\first@char A\else
      \def\first@char{}%
    \fi
    \macc@nested@a\relax111{\first@char}%
  \fi
  \endgroup
}

\def\house#1{{%
    \setbox0=\hbox{$#1$}
    \vrule height \dimexpr\ht0+1.4pt width .4pt depth \dp0\relax
    \vrule height \dimexpr\ht0+1.4pt width \dimexpr\wd0+2pt depth \dimexpr-\ht0-1pt\relax
    \llap{$#1$\kern1pt}
    \vrule height \dimexpr\ht0+1.4pt width .4pt depth \dp0\relax
}}

\def\House#1{{%
    \setbox0=\hbox{$#1$}
    \vrule height \dimexpr\ht0+1.4pt width .4pt depth \dp0\relax
    \vrule height \dimexpr\ht0+1.4pt width \dimexpr\wd0+2pt depth \dimexpr-\ht0-1pt\relax
    \llap{$#1$\kern1pt}
    \vrule height \dimexpr\ht0+1.4pt width .4pt depth \dp0\relax
}}

\def\housealp{{%
    \setbox2=\hbox{$\alpha$}
    \vrule height \dimexpr\ht2+1.75pt width .4pt depth \dp2\relax
    \vrule height 6.05pt width \dimexpr\wd2+2pt depth -5.65pt
    \llap{$\alpha$\kern1pt}
    \vrule height \dimexpr\ht2+1.75pt width .4pt depth \dp2\relax
}}

\newcommand{\xdasharrow}[2][->]{
\tikz[baseline=-\the\dimexpr\fontdimen22\textfont2\relax]{
\node[anchor=south,font=\scriptsize, inner ysep=1.5pt,outer xsep=2.2pt](x){#2};
\draw[shorten <=3.4pt,shorten >=3.4pt,dashed,#1](x.south west)--(x.south east);
}
}

\newcommand{\eqn}[1]{\begin{equation*}#1\end{equation*}}

\newcommand{\aDiv}{\widehat{\Div}}

\newtheorem{thm_}{Theorem}[section]
\newtheorem{lemma_}[thm_]{Lemma}
\newtheorem{prop_}[thm_]{Proposition}
\newtheorem{conj_}[thm_]{Conjecture}
\newtheorem{cor_}[thm_]{Corollary}
\theoremstyle{definition}
\newtheorem{eg_}[thm_]{Example}
\newtheorem{def_}[thm_]{Definition}
\newtheorem{rk_}[thm_]{Remark}
\newtheorem{qu_}[thm_]{Question}

\title{Arakelov geometry of toric bundles: Okounkov bodies and BKK}
\author{Nuno Hultberg}


\address{Nuno Hultberg. University of Copenhagen, Institute of Mathematics, Universitetsparken 5, 2100 Copenhagen, Denmark;
ORCiD: \href{https://orcid.org/0000-0003-0097-0499}{orcid.org/0000-0003-0097-0499}
}
\email{nh@math.ku.dk}

\calclayout
\begin{document}

\maketitle

\begin{abstract}
This article introduces the study of toric bundles and the morphisms between them from the perspective of adelic fibre bundles, as introduced by Chambert-Loir and Tschinkel. We study the Okounkov bodies and Boucksom-Chen transforms of suitable adelic line bundles on toric bundles. Finally, we prove an arithmetic analogue of a formula for intersection numbers due to Hofscheier, Khovanskii and Monin. We apply this to the study of compactifications of semiabelian varieties, whose height and successive minima we compute. This extends computations of Chambert-Loir to arbitrary toric compactifications.
\end{abstract}

\freefootnote{An earlier version of this article was included in the author's PhD thesis. The main difference is a generalization of Theorem \ref{thm:intersection_formula} beyond the case of a regular base scheme.}

\section{Introduction}

Let $G$ be a semiabelian variety over a field $\K$. A semiabelian variety is canonically an extension of an abelian variety $A$ by a torus $\T$:
\[\label{ses}
\begin{tikzcd}
0 \arrow[r] & \bbT  \arrow[r, "i"] & G \arrow[r, "\pi"] & A \arrow[r] & 0.
\end{tikzcd}
\]
After possibly passing to a finite extension of $\K$, we assume that $\T$ is split, i.e.\ $\bbT\cong \bbG_m^t$. We refer to $\T$ as the torus part of $G$ and to $A$ as its abelian quotient. Abstractly, we may view $G$ as a $\T$-torsor over $A$. This perspective clarifies the canonical identification $\Ext^1_{\K}(A,\bbG_m) \cong A^\vee(\K)$ given by the Weil-Barsotti formula, see \cite[\S III.8]{Oort_commutative_group_schemes}.

Many methods of algebraic and arithmetic geometry rely on the properness of the studied varieties. If $\T$ is not trivial, $G$ is not proper. We obtain a compactification $\ov{G}$ of $G$ by considering it as a $\T$-torsor and applying a toric compactification $X$ of $\T$. This naturally places us in the setting of toric bundles. We concretize the constructions of \cite{chambert-loir_tschinkel_arithmetic_torsors} in this setting. Details of the upcoming discussion are contained in Section \ref{sec:toric_bundles}.

Let $\B$ be a variety over $\K$ and let $\sT$ be a $\T$-torsor over $B$, for a split torus $\T$. Let $N= \Hom(\bbG_m,\T)$ be the lattice of co-characters and $M = \Hom(\T,\bbG_m)$ the lattice of characters of $\T$. A fan $\Sigma$ in $N_\bbR$ defines a toric variety $\toricvar_\Sigma$ containing $\T$. We define a toric bundle $\toricbun_\Sigma = (\sT \times X_\Sigma)/\T$ by using Zariski descent, where $x \in \T$ acts by $(x,x^{-1})$. The fibres of $\toricbun_\Sigma \to \B$ can be identified with $\toricvar_\Sigma$. Such an identification is canonical up to the action of $\T$. Furthermore, a $\T$-invariant Cartier divisor $D$ on $\toricvar_\Sigma$ gives rise to a $\T$-invariant Cartier divisor $\rho(D)$ on $\toricbun_\Sigma$. More precisely, there is a map $\rho: \Div_\T(\toricvar_\Sigma) \to \Div_\T(\toricbun_\Sigma)$ which after restricting to a fibre identified with $\toricvar_\Sigma$ yields the identity.

The torus bundle $\sT$ induces a homomorphism $c:M \to \Pic(B)$ by sending $m \in M$ to $(\sT \times \bbG_m)/\T$, where the action of $x\in\T$ is by $(x,\chi^{-m}(x))$. In fact, this defines an isomorphism between the group of isomorphism classes of $\T$-torsors and $\Hom(M,\Pic(\B))$. If $K$ is endowed with an absolute value or is a global field, we define metrized and adelic torus bundles. Their isomorphism classes will be isomorphic to $\Hom(M,\mPic(\B))$ and $\Hom(M,\aPic(\B))$ respectively. Given a metrized/adelic torus bundle $\aTbun$ with underlying torus bundle $\Tbun$, we define analogous constructions $\mrho:\mDiv_\T(\toricvar_\Sigma) \to \mDiv_\T(\toricbun_\Sigma)$ and $\arho:\aDiv_\T(\toricvar_\Sigma) \to \aDiv_\T(\toricbun_\Sigma)$. We denote the homomorphism describing its isomorphism class by $\ac:M \to \aPic(\B)$.

We may study the group structure on $G$ by exhibiting it as a map of toric bundles, a notion we will introduce. Using the theory of toric bundles we can extend maps on $G$ to maps to its compactifications. For instance, let $[n]$ denote the multiplication by $n$ on $G$ and $G_\Sigma$ denote the compactification of $G$ as a toric bundle associated to the fan $\Sigma$. Then, the morphism $[n]$ extends to an endomorphism $[n]_\Sigma$ of $G_\Sigma$.  Furthermore, the multiplication on $G$ extends to an action of $G$ on $G_\Sigma$. The multiplication $[n]_\Sigma$ on $G_\Sigma$ does not give rise to a polarized dynamical system unless $G$ is a torus or an abelian variety. Instead, canonical heights are defined in terms of a toric contribution $\bar{\sL}$ and a contribution from the abelian variety $\bar{\pi}^* \bar{\sM}$. Here $\bar{\sM}$ is a canonically metrized ample symmetric line bundle on the abelian variety $A$. The toric contribution is defined by endowing Cartier divisors of the form $\rho(D)$ with canonical metrics with respect to $[n]_\Sigma$, which we denote by $\rho(D)^{\can}$. We will prove that $\rho(D)^{\can}$ is isometric to $\arho(D^{\can})$, when endowing the line bundles on the abelian variety with their canonical metric. This framework enables us to conceptualize the computations in \cite{cl2}, which serves as a major inspiration for this work.

\subsection{Statement of results}
We say that $K$ is a global field if it is a number field or the function field of a geometrically irreducible curve $\sS$ over a field $k$. In the number field case, we denote by $\sS$ the spectrum of the ring of integers. For details on notation we refer to Section \ref{sec:prelim}.

Let $\aTbun$ be an integrable torus bundle (cf.\ Definition \ref{defi:torus_bundle}) over a proper variety $\B$ over a global field $\K$ with associated homomorphism $\ac:M \to \aPic(\B)$. Let $\Delta\subset M_\bbR$ be a rational polytope and $\Sigma$ a complete rational fan that refines the normal fan of $\Delta$. Denote the toric bundle defined by these data by $\toricbun$. Let $D$ be the toric Cartier divisor defined by $\Delta$. Let $\ov{\Delta}$ denote the datum of a polytope $\Delta$ together with a collection $(\theta_v)$ of local roof functions on $\Delta$ with global roof function $\theta$. It gives rise to a semipositive toric Cartier divisor $\ov{D}$. We define $\rho(\Delta) \defeq \rho(D)$ and $\arho(\ov{\Delta}) \defeq \arho(\ov{D})$.

We briefly recall the notion of Zhang minima. Let $\bar{L}$ be an adelically metrized line bundle on a geometrically irreducible variety $X$ of dimension $d$. For $\lambda \in \bbR$, define the set $X_{\ov{L}}(\lambda)$ to be the Zariski closure in $X$ of the set $\{ x \in X(\Kbar)\mid h_{\ov{L}}(x) \leq \lambda\}$.

\begin{definition}
    The $i$-th \emph{Zhang minimum} of $\bar{L}$ is defined to be
    \begin{equation}
        \zeta_i(\ov{L}) = \inf \{ \lambda\mid  \dim X_{\ov{L}}(\lambda) \geq i-1\}
    \end{equation}
    in \cite{Zhang_thesis_inequality}. The \emph{absolute minimum} $\zeta_{\abs}(\ov{L})$ is defined to be the first Zhang minimum, i.e.\ the infimum over the $\lambda$ such that $X_{\ov{L}}(\lambda) \neq \emptyset$. The \emph{essential minimum} is the $(d+1)$-st Zhang minimum, i.e.\ the infimum over the $\lambda$ such that $X_{\ov{L}}(\lambda) = X$.
\end{definition}

We compute the essential and absolute minima of adelically metrized line bundles on toric bundles  by studying their Okounkov body and Boucksom-Chen transform with respect to a suitable flag. 

\begin{restatable}{rthm}{minima}\label{thm:minima}
    Let $\ov{L}$ be an adelic line bundle on $\B$ such that $\ov{L} + \ac(m)$ is geometrically big for some $m \in \Delta$. Let $\Delta^\circ$ denote the interior of $\Delta$. Then, we have the following formula for the essential minimum of $\arho(\ov{\Delta}) + \pi^*\ov{L}$ on $\toricbun$:
    \[
    \zeta_{\ess}(\arho(\ov{\Delta}) + \pi^*\ov{L}) = \sup_{m \in \Delta} \left\{\zeta_{\ess}(\ov{L} + \ac(m))+\theta(m)\right\}.
    \]
    If in addition $\ov{L} + \ac(m)$ is semipositive for all $m \in \Delta$,
    \[
    \zeta_{\abs}(\arho(\ov{\Delta}) + \pi^*\ov{L}) = \inf_{m \in \Delta^\circ} \left\{\zeta_{\abs}(\ov{L} + \ac(m))+\theta(m)\right\}.
    \]
\end{restatable}

\begin{remark}
    One may replace $\inf_{m \in \Delta^\circ}$ by $\inf_{m \in \Delta}$ provided $\ov{L} + \ac(m)$ is geometrically big for all $m \in \Delta$.
\end{remark}

We are able to compute the height filtration of a toric bundle in terms of data from the base variety, see Section \ref{sec:suc_min}. The height filtration has previously been computed explicitly for toric varieties in \cite{sucmin} and in the case of flag varieties over function fields in \cite{fan2024heightfiltrationsbaseloci}. Without such kind of additional structure, the problem of computing height filtrations seems very hard.

In addition to computing the essential and absolute minima, we seek a combinatorial formula for arithmetic intersection numbers on toric bundles inspired by \cite{toricbundles}. We will refer to this as the arithmetic bundle BKK theorem. It is a common generalization of the arithmetic BKK theorem \cite[Theorem 5.2.5]{toricheights} and the bundle BKK theorem \cite[Theorem 4.1]{toricbundles}. The study of Okounkov bodies used in the proof of Theorem \ref{thm:minima} suffices to obtain the arithmetic bundle BKK theorem for complete intersection cycles. To prove it for arbitrary cycles requires other methods.

Let $\atoricbun$ be an adelic integrable projective toric bundle of relative dimension $t$ with smooth generic fibre over a smooth projective base variety $\B$ of dimension $g$. Let $\sB$ be a flat, projective model of $\B$ over $\sS$. Let $A^*(\blank)$ denote the operational arithmetic Chow cohomology introduced in Section \ref{subsec:chow_homology}. Let $\gamma \in A^{g+1-i}(\sB)$ and denote by $[\infty]\in \ArChow^1(\sB)_\bbR$ the class of a trivial Cartier divisor endowed with constant Green's functions at all places such that $h_{[\infty]}(x) = 1$ for all $x \in \B(\Kbar)$.

\begin{restatable}{rthm}{arithmeticbundlebkk}
    \label{thm:intersection_formula}
    The intersection numbers below are well-defined and the following identity holds:
    \[
    i!\arho(\ov{\Delta})^{t+i}\pi^*\gamma = (t+i)!\int_{\Delta} (\ac(m)+\theta(m)[\infty])^i\gamma dm.
    \]
\end{restatable}

We remark that the formula can be polarized, see \cite[Equation 7.8]{greenberg_forms_in_many_variables}, to obtain a description of the intersection number $\arho(D_1)\dots\arho(D_{t+i}) \pi^*\gamma$ for integrable toric Cartier divisors $D_1,\dots, D_{t+i}$. By approximation, one can replace $\gamma$ by a product $\gamma'\ac_1(L_1)\dots\ac_1(L_l)$ for integrable line bundles $L_1, \dots, L_l$ on $\B$ and $\gamma \in A^{g+1-i-l}(\sB)$.

\begin{remark}
    If $\sB$ is regular there is a map $\ArChow^*(\sB) \to A^*(\sB)$. For general $\sB$, characteristic classes of hermitian vector bundles will give rise to elements in $A^*(\sB)$. In Section \ref{subsec:chow_homology}, we discuss an alternative to $A^*(\sB)$ that is homological in nature given by $b$-cycles satisfying positivity conditions.
\end{remark}

We would like to note that the bundle BKK theorem of \cite{toricbundles} is used to compute the cohomology ring of smooth toric bundles. This relies on Poincar\'e duality for oriented manifolds and Leray's theorem on the cohomology of fibrations. It would be interesting to find a similar description for a suitable subring of the arithmetic Chow group of a toric bundle or for some equivariant arithmetic Chow group.

Let $G$ be a semiabelian variety over a global field $\K$ with abelian quotient $A$, given as a torus bundle by a map $c:M \to A^\vee(\Kbar)$. Let $\sM$ denote an ample symmetric line bundle on $A$. Endowing it with the canonical metrics gives it the structure of an adelically metrized line bundle $\ov{\sM}$. Let $\widehat{h}$ denote the N\'eron-Tate height on $A(\Kbar)\otimes \bbR$. We note that it factors through the polarization $A(\Kbar)\otimes \bbR \to A^\vee(\Kbar)\otimes \bbR$. We denote its evaluation on $A^\vee(\Kbar)\otimes \bbR$ by $\widehat{h}$ as well. Let $\ov{G}$ be the compactification of $G$ with respect to a fan in $M_\bbR$. Let $D$ be an ample toric Cartier divisor on $\toricvar_\Sigma$ with Newton polytope $\Delta \subset M_\bbR$. Let $\sF(\Delta)^{i}$ denote the set of $i$-dimensional faces of $\Delta$.

\begin{restatable}{rthm}{heightsemiabelian}\label{thm:height_formula_semiabelian}
    The height of a compactified semiabelian variety $\ov{G}$ can be computed as
    \[
        h_{\rho(D)^{\can}\otimes \pi^*\bar{\sM}}(\ov{G}) = -(d+1)!\int_{\Delta} \widehat{h}(c(m)) dm.
    \]
\end{restatable}

\begin{restatable}{rthm}{minimasemiabelian}\label{thm:successive_minima_formula_semiabelian}
    The $i$-th successive minimum of $\bar{G}$ with respect to $\rho(D)^{\can}\otimes \pi^*\bar{\sM}$ satisfies $\zeta_i(\bar{G}) = \zeta_{1}(\bar{G})$ for $i\leq g+1$. For $i \geq g+1$,
    \[
    \zeta_{i}(\bar{G}) = -\max_{F \in \sF(\Delta)^{t+g+1-i}} \min_{m \in F} \widehat{h}(c(m)).
    \]
\end{restatable}

By specializing $\Delta$ in the above theorems, one can recover the results of Section 4 in \cite{cl2}.

\subsection{Organization of the article}

Section \ref{sec:prelim} is devoted to establishing notations and recalling known facts on toric varieties and Okounkov bodies in algebraic and arithmetic geometry. We further explain how to associate intersection numbers to an arbitrary arithmetic cycle and a collection of integrable divisors.

In Section \ref{sec:toric_bundles} we introduce metrized toric bundles and construct the map $\arho$. We proceed to study their basic properties.

In Section \ref{sec:okounkov}, we compute Okounkov bodies and the Boucksom-Chen transform for line bundles on toric bundles. We apply this to prove Theorem \ref{thm:minima}.

We prove the arithmetic bundle BKK theorem in Section \ref{hkm} by means of arithmetic convex chains.

Section \ref{sec:examples} illustrates the results from the previous sections with examples in the realm of semiabelian varieties. In particular, we prove Theorem \ref{thm:height_formula_semiabelian} and \ref{thm:successive_minima_formula_semiabelian}. We then apply this to recover computations of Chambert-Loir.

\subsection*{Acknowledgements}
I thank my advisor Lars K\"uhne for his guidance, discussions and comments on earlier drafts. I thank Jos\'e Burgos Gil and Klaus K\"unnemann for discussions and comments. In particular, for Klaus K\"unnemann's suggestion to consider non-regular base schemes in Theorem \ref{thm:intersection_formula}. I thank Bror Hultberg for discussions on convex analysis and Fabien Pazuki for comments on an early draft. I lastly thank Wenbin Luo and Martin Sombra for short, but helpful conversations.

\section{\label{sec:prelim}Preliminaries}

\subsection{\label{subsec:toric_varieties_prelim}Arakelov geometry of toric varieties}
We will assume basic familiarity with toric varieties, but still give a brief recollection on some basic facts. We then present the extension of these facts in arithmetic geometry as developed in \cite{toricheights}, which we recommend for an in-depth treatment of the Arakelov geometry of toric varieties. In addition, we introduce the original notion of an adelic polytope. This turns out to be convenient to transfer arguments from the classical theory.

Let $\T$ be a split torus over a field $K$.
\begin{definition}
    A \emph{toric variety} with torus $\T$ is a normal variety $X$ with a dense open embedding $\T\subseteq X$ and an action $\T\times X \to X$ extending multiplication on $\T$.
\end{definition}

We follow the convention to denote by $N$ the set of cocharacters $\Hom(\bbG_m,\T)$ and by $M$ the set of characters $\Hom(\T,\bbG_m)$. They are finitely generated free abelian groups, dual to one another. Toric varieties with torus $\T$ are in bijection to rational fans on $N_\bbR$. We denote by $X_\Sigma$ the toric variety associated to a fan $\Sigma$.

\begin{definition}
    A \emph{virtual support function} or \emph{virtual polytope} with respect to a fan $\Sigma$ on $N_{\bbR}$ is a function $|\Sigma| \to \bbR$ that is linear on each cone in $\Sigma$. The set of virtual support functions with respect to $\Sigma$ is denoted by $\calP_\Sigma$.
\end{definition}

Virtual polytopes are generalizations of polytopes by Legendre-Fenchel duality. Consider the monoid of polytopes $\calP^+_\Sigma$ in $M_{\bbR}$ whose normal fan coarsens $\Sigma$ with addition given by Minkowski sum. The \emph{normal fan} of a polytope $\Delta$ is defined by associating to each $i$-dimensional face $F$ the $(n-i)$-dimensional cone
\[
\left\{ n \in N_{\bbR}\left| \forall u \in F:\ \langle n,u\rangle = \sup_{x \in \Delta} \langle n,x\rangle \right.\right\}.
\]
The normal fan is the collection of these cones.

The group completion of $\calP^+_\Sigma$ can be identified with $\calP_\Sigma$ through Legendre-Fenchel duality. Under this duality polytopes are identified with concave functions.

\begin{definition}
    We denote by $\calP$ and $\calP^+$ the set of (virtual) polytopes with respect to any rational fan.
\end{definition}

\begin{definition}
    A \emph{toric Cartier divisor} on a toric variety $X$ is defined to be a Cartier divisor which is invariant under the action of the torus $\mu:\bbT\times X \to X$. This means that a Cartier divisor $D$ is toric if $\mu^*D = \pi^*D$, where $\pi$ denotes the projection map. Denote the set of toric Cartier divisors on a toric variety $X$ by $\Div_\bbT(X)$. A toric $\bbR$-Cartier divisor is an element of $\Div_\bbT(X)_\bbR = \Div_\bbT(X) \otimes \bbR$.
\end{definition}

\begin{theorem}[Section 3.3 \cite{toric}]
    There is an isomorphism $\rho:\calP_\Sigma \to \Div_\bbT(X_{\Sigma})_\bbR$.
\end{theorem}

Let $X$ be a finite type scheme over a field $K$ endowed with an absolute value. Let $\widehat{K}$ denote its completion. Then, we define $X^{\an}$ to be the analytification of $X_{\widehat{K}}$ in the sense of Berkovich, introduced in \cite{berk}. Suppose $X_{\widehat{K}}=\Spec(A)$ is affine. Then, $X^{\an}$ as a set can be identified with
\[
    \{|\blank|:A\to \bbR \text{ multiplicative seminorm extending the norm on } \widehat{K}\}.
\]

For a split torus $\T$ we can define the tropicalization map by
\begin{align*}
    \bbT^{\an} &\to N_{\bbR}\\
    x&\mapsto (m \mapsto |\chi^m(x)|).
\end{align*}
We will not need to introduce tropicalization in the general setting.

\begin{definition}
    Let $D$ be a Cartier divisor on a variety $X$ over a field $K$ endowed with an absolute value. Then, a continuous Green's function for $D$ is a function $g:(X\setminus \Supp D)^{\an} \to \bbR$ such that for each $U \subseteq X$ on which $D$ is defined by a section $f$, the function $g+\log|f(x)|:(U\setminus \Supp D)^{\an} \to \bbR$ extends to a continuous function on $U^{\an}$.
\end{definition}

\begin{definition}
    Let $K$ be a field with an absolute value. A \emph{metrized Cartier divisor} $\bar{D}=(D,g)$ on a proper variety $X$ consists of a Cartier divisor $D$ and a continuous $D$-Green's functions $g$.
    
    If $K$ is a global field, an \emph{adelic Cartier divisor} $\bar{D}=(D,g_v)$ consists of a Cartier divisor and continuous Green's functions $g_v$ for all places $v \in M_K$ such that there exists a dense open subset $U \subseteq \sS$ and a normal proper model $(\sX_U,D_U)$ over $U$ such that for all $v \in U$, $g_v$ is induced by the model. An adelic Cartier divisor is called effective if the underlying Cartier divisor is effective and $g_v \geq 0$ for all places $v$.

    These notions have an analogue on the level of line bundles, namely metrized line bundles and adelically metrized line bundles. We denote the set of adelic Cartier divisors by $\aDiv(X)$.
\end{definition}

For the purposes of arithmetic intersection this is too general. We restrict to \emph{integrable} metrics at each place. An integrable metric is the difference of semipositive metrics. \emph{Semipositive} metrics are limits of model/smooth metrics satisfying certain positivity conditions.

Let $X$ be a proper variety of dimension $d$. Then, there is an intersection pairing defined in \cite{zhang_small_points_adelic_metrics} that to $d+1$ integrable divisors $\ov{D}_0,\dots,\ov{D}_{d}$ associates an intersection number $\ov{D}_0\dots\ov{D}_{d}\in \bbR$. This pairing factors through the group of integrable adelic $\bbR$-Cartier divisors, see \cite[Section 3.2]{ballay_nakai_moishezon_R-Cartier}. The group of adelic $\bbR$-Cartier divisors $\aDiv(X)_\bbR$ consists of an $\bbR$-Cartier divisor and compatible Green's functions at every place. It admits a natural surjection $\aDiv(X)\otimes\bbR \to \aDiv(X)_\bbR$. Given an adelic $\bbR$-Cartier divisor $\ov{D}$ we denote its top intersection product $\ov{D}^{d+1}$ by $h_{\ov{D}}(X)$ and call it the height of $X$ with respect to $\ov{D}$. This convention differs from \cite{cl2}, where the height is normalized by a factor $\frac{1}{(\dim X + 1)\deg_D(X)}$, but agrees with the convention of \cite{F_Ballay_Succesive_minima}.

When $X$ is a toric variety we call a Green's function $g$ toric if it factors through tropicalization on the underlying torus, i.e.\ is invariant under the action of the unit torus. We will denote the set of toric adelic divisors by $\aDiv_\bbT(X)$. The toric dictionary extends to this setting. The continuous metrics on $\Psi \in \calP_\Sigma$ are in bijection to continuous functions $\psi$ on $N_\bbR$ such that $\psi - \Psi$ is bounded. Under this bijection, semipositive metrics correspond to concave functions $\psi$, see \cite[Theorem 4.8.1]{toricheights}.

It will be useful to relate integrable divisors directly to polytopes.

\begin{definition}\label{defi:I_metrized_polytope}
    Let $J$ be a finite set and $V$ a finite-dimensional vector space. A $J$-metrized polytope is a polytope $\ov{\Delta}\subset V\oplus \bigoplus_{j\in J}\bbR$ that can be obtained via the following construction. 
    
    Let $(\Delta,(\theta_j)_{j \in J})$ consisting of a polytope $\Delta$ in $V$ and for each $j \in J$, a concave function $\theta_j:\Delta \to \bbR_{\geq 0}$. The associated polytope is
    \[
        \ov{\Delta} = \{(x,t_j)\in V\oplus \bigoplus_{j\in J}\bbR\mid  x \in \Delta, 0\leq t_j \leq \theta_j(x)\}.
    \]
    Denote the set $J$-metrized polytopes by $\calP^{J,+}$.
    
    For an infinite set $I$, we define the set of $I$-metrized polytopes to be the filtered colimit
    \[
        \underset{J\subset I \text{ finite}}{\colim} \calP^{J,+},
    \]
    where the transition map for $J \subset J'$ is given by $\ov{\Delta} \mapsto \ov{\Delta} \times \prod_{j\in J'\setminus J} 0$. This is compatible with the monoid structure, hence giving a monoid of $I$-metrized polytopes.

    We denote the group completion of $\calP^{I,+}$ by $\IP$ and call it the group of $I$-metrized virtual polytopes. When the set $I$ is the set of places $M_K$ of a global field $K$ we will use the notation $\aP^+$ and $\aP$. Their elements will be referred to as (virtual) adelic polytopes.

    For $I=\{*\}$, we speak simply of metrized polytopes.
\end{definition}

\begin{definition}\label{defi:roof_functions_polytopes}
    An $I$-metrized polytope $\ov{\Delta}$ has an associated underlying polytope $\Delta \subset V$ and for each $i\in I$ a \emph{local roof function} $\theta_i:\Delta \to \bbR$ satisfying $\theta_i =0$ for almost all $i$ such that $\ov{\Delta}$ is associated to $(\Delta,(\theta_j)_{i \in I})$ as described in Definition \ref{defi:I_metrized_polytope}. The \emph{global roof function} is defined as the finite sum $\theta = \sum_{i \in I}\weighti\theta_i$, where $\weighti$ is a choice of weights that is clear from the context. In particular, for adelic divisors $\weighti$ will be $1$ for all $i$ if $K$ is a function field and $\frac{[K_i:\bbQ_i]}{[K:\bbQ]}$ if $K$ is a global field.

    For each $i \in I$, we associate a local metrized polytope 
    \[
        \Delta_i = \{(x,t)\in V\oplus \bbR\mid  x \in \Delta, 0\leq t \leq \theta_i(x)\}.
    \]

    We define the global polytope $\widehat{\Delta}$ to be
    \[
        \{(x,t)\in V\oplus \bbR\mid  x \in \Delta, 0\leq t \leq \theta(x)\}.
    \]
\end{definition}

Let us recall \cite[Theorem 4.8.1]{toricheights} using this new language. There is an isomorphism of monoids between semipositive effective metrized divisors and metrized polytopes. This globalizes to an isomorphism of monoids between semipositive effective adelic divisors and adelic polytopes. By group completion, the isomorphisms extend to isomorphisms between virtual (adelically) metrized polytopes and integrable (adelically) metrized divisors.

Model metrics will play a crucial role in the proof of the arithmetic bundle BKK-theorem.

Let $K$ be a complete field with respect to an absolute value associated to a non-trivial discrete valuation. Let $\Sigma$ be a complete rational fan on $M_\bbR$. Then, the set of toric models can be identified with rational fans $\widetilde{\Sigma}$ in $N_\bbR \oplus \bbR_{\geq 0}$ whose intersection with $N_\bbR \oplus 0$ is $\Sigma$, cf. \cite[Theorem 3.5.3]{toricheights}. We refer to the fan $\Sigma^{\can}$ consisting of the cones of the form $\sigma \oplus 0$ and $\sigma \oplus \bbR_{\geq 0}$ for $\sigma \in \Sigma$ as the \emph{canonical (metrized) fan} associated to $\Sigma$. Denote the toric model associated to $\widetilde{\Sigma}$ by $\sX_{\widetilde{\Sigma}}$. The following follows from \cite[Theorem 3.6.7]{toricheights}.

\begin{theorem}
    The set of semipositive effective $\bbR$-divisors on $\sX_{\widetilde{\Sigma}}$ is in bijection to metrized polytopes whose normal fan restricted to $N_\bbR \oplus \bbR_{\geq 0}$ coarsens $\widetilde{\Sigma}$.
\end{theorem}

\begin{proof}
    Piecewise affine concave functions $\psi$ on $\widetilde{\Sigma} \cap (N_\bbR \times \{ 1\})$ are in bijection with semipositive $\bbR$-divisors on $\sX_{\widetilde{\Sigma}}$ by a combination of \cite[Theorem 3.6.7]{toricheights} and \cite[Theorem 3.7.3]{toricheights}. The divisor corresponding to $\psi$ is effective precisely if $\psi$ is nonnegative. Let $\Psi$ be the recession function of $\psi$. We may extend it to a conical function on $\widetilde{\psi}:N_\bbR \oplus \bbR$ by setting $\psi(n,x) = \Psi(n)$ for $x<0$. Legendre-Fenchel duality restricts to a correspondence between metrized polytopes whose normal fan coarsens $\widetilde{\Sigma}$ and conical functions of the form above.
\end{proof}

We can consider metrics on toric divisors compatible with $\widetilde{\Sigma}$ even if $K$ does not have a discrete valuation.

\begin{definition}
    A metrized divisor is compatible with $\widetilde{\Sigma}$ if it is a difference of two divisors associated to metrized polytopes whose normal fan restricted to $N_\bbR \oplus \bbR_{\geq 0}$ coarsens $\widetilde{\Sigma}$. We will denote the set of metrized polytopes compatible with $\widetilde{\Sigma}$ by $\calP_{\widetilde{\Sigma}}$. It can be identified with conical functions $\psi:N_\bbR \oplus \bbR \to \bbR$ that is linear on cones of $\widetilde{\Sigma}$ and such that $\psi(n,x) = \psi(n,0)$ for $x<0$. 
\end{definition}

\begin{definition}
    We define an adelic fan $\widetilde{\Sigma}$ to be a collection of fans $\widetilde{\Sigma}_v$ in $N_\bbR \oplus \bbR_{\geq 0}$ for each $v \in M_K$ such that almost all the fans are canonical. A (virtual) adelic polytope is said to be compatible with $\widetilde{\Sigma}$ if it is compatible with each $\widetilde{\Sigma}_v$ when considered as a metrized polytope. The set of such polytopes will be denoted $\aP_{\widetilde{\Sigma}}$ and $\aP^+_{\widetilde{\Sigma}}$.
    
    One can identify virtual adelic polytopes with conical functions $\psi:N_\bbR \oplus \bigoplus_{i \in I}\bbR \to \bbR$ that factors over $N_\bbR \oplus \bigoplus_{j \in J}\bbR$ for some finite $J \subset M_K$ satisfying that $\psi(n,(x_j)) = \psi(n,(\max\{0,x_j\}))$ and linear on each cone of $\Sigma$ and each cone of the form $\sum_j\in J \sigma_j$, where $\sigma_j \cap \N_\bbR = \sigma$ for some fixed cone $\sigma \in \Sigma$.

    The set of cones of $\aP_{\widetilde{\Sigma}}$ is defined to be the union of $\Sigma$ and the cones of $\aP_{{\widetilde{\Sigma}}_v}$ at each place $v$. Here the cones at $v$ contained in $N_\bbR \times 0$ are identified with $\Sigma$.
\end{definition}

\begin{definition}
    Let $v \in M_K$ be a place. Denote by $\aP^v$ the set of virtual adelic polytopes canonical at all places $w\neq v$, i.e.\ the ones that come from $\{v\}$-metrized polytopes in the colimit. Then, we say that $\ov{\Delta}\in\aP_{\widetilde{\Sigma}}^+$ is $v$-interior if its is in the interior of $\left(\ov{\Delta} + \aP^v\right) \cap \aP_{\widetilde{\Sigma}}^+ \subset \ov{\Delta} + \aP^v$.
\end{definition}

\subsection{\label{prelim_okounkov}Okounkov bodies and roof functions}
Let $X$ be a projective variety over a global field $K$ and let $\ov{D}$ be a geometrically big adelic Cartier divisor on $X$. The concave transform of $\ov{D}$ as defined in \cite[Definition 1.7]{Boucksom_Chen_Okounkov_bodies_of_filtered_linear_series} is a concave function on the Okounkov body $\Delta$ associated to the underlying line bundle encoding information on the adelically metrized line bundle in terms of convex geometry. It is defined in terms of the filtered linear series associated to $\ov{D}$.

Consider the graded linear series $V^\bullet = \bigoplus_{n=0}^\infty H^0(X,\sO(n\ov{D}))$. Let $x \in X(\Kbar)$ be a regular point and fix an isomorphism $\z:\widehat{\sO}_{X,x} \iso \Kbar[[x_1, \dots, x_d]]$. This induces a rank $d$ valuation $\nu_n$ on each $H^0(X,\sO(n\ov{D}))$ by taking the valuation on $\Kbar[[x_1, \dots, x_d]]$ sending $x^{a_1}_1\cdot\dots\cdot x^{a_d}_d$ to $(a_1,\dots,a_d)$ and on a linear combination of monomials the lexicographically smallest term. The Okounkov body of $D$ is defined to be as the closure of
\[
\bigcup_{n=0}^\infty \frac{1}{n} \nu(H^0(X,\sO(n\ov{D}))).
\]
This is a convex body and we denote it by $\Delta(D)$, where the choice of $\z$ is understood. It is invariant under numerical equivalence by \cite[Proposition 4.1]{lazmus}. Just as easily we can define Okounkov bodies for subalgebras.

Each term in the graded algebra is endowed with a filtration by minima. More precisely, on an adelically normed vector space $V$ we may define the $F^t V$ to be the sub-vector-space generated by vectors of height $\leq t$, see \cite[Definition 3.1]{Boucksom_Chen_Okounkov_bodies_of_filtered_linear_series}. We define $F^t V^\bullet = \bigoplus F^{nt} H^0(X,\sO(n\ov{D}))$. The associated Okounkov body will be denoted by $\Delta^t(\ov{D})$. Boucksom and Chen in \cite{Boucksom_Chen_Okounkov_bodies_of_filtered_linear_series} define the concave transform $G_{\sL,\z}: \Delta(D) \to \bbR$ by $x \mapsto \inf \{ t\mid x\in \Delta^t(\ov{D})\}$, which is concave and upper semicontinuous on the boundary. The hypograph of this function is called the arithmetic Okounkov body.

A common way to obtain an isomorphism $\widehat{\sO}_{X,x} \iso \Kbar[[x_1, \dots, x_d]]$ is by fixing a flag
\[
X_\bullet:X\supset X_1 \supset X_2 \supset\dots \supset X_d =\{x\},
\]
of irreducible subvarieties $X_i$ of codimension $i$ that are non-singular at $x$. If $\z$ is induced by $X_\bullet$, we denote the Boucksom-Chen transform by $G_{\ov{D},X_\bullet}$.

The concave transform defined above encodes important information on the adelically metrized line bundle. For instance, Balla\"y shows that the essential and the absolute infimum of an adelic line bundle are determined by its associated filtered linear series in the semipositive case. \cite[Theorem 1.7]{Arithmetic_Demailly_Qu_Yin} allows to remove one semipositivity assumption. Here is a version of \cite[Proposition 7.1]{F_Ballay_Succesive_minima} adapted accordingly together with \cite[Corollary 1.2]{ballay:hal-03606531}.

\begin{theorem}\label{thm:ballay}
    Let $\ov{D} \in \widehat{\Div}_{\bbR}$ with $D$ big. Then, 
\[\zeta_{\ess}(\ov{D}) = \max_{\alpha \in \Delta(D)}G_{\ov{D}}(\alpha).\]
If $\ov{D}$ is semipositive, \[\zeta_{\abs}(\ov{D}) = \inf_{\alpha \in \Delta(D)}G_{\ov{D}}(\alpha).\]
\end{theorem}

Furthermore, the Okounkov body and the concave transform are related to the volume and the arithmetic volume respectively. We will not discuss arithmetic volume functions as we will only use its relation to heights via the arithmetic Hilbert-Samuel theorem. The following result contains a geometric part \cite[Theorem A]{lazmus} and an arithmetic one taken from \cite[Theorem 6.4]{F_Ballay_Succesive_minima}, but which is implicit in \cite{Boucksom_Chen_Okounkov_bodies_of_filtered_linear_series}. We apply the normalization of arithmetic volumes in \cite{Boucksom_Chen_Okounkov_bodies_of_filtered_linear_series} which differs from the one in \cite{F_Ballay_Succesive_minima}.

\begin{theorem}\label{thm:arithmetic_volume}
    Let $D$ be a big $\bbR$-Cartier divisor on $X$. Then,
    \[
    \vol(D) = d!\vol_{\bbR^d}(\Delta(D))
    \]
    If $\ov{D}$ is an adelic divisor whose underlying divisor is $D$, then
\[\widehat{\vol}(\ov{D}) = (d + 1)!\int_{\Delta(D)}\max\{0, G_{\ov{D}}\}d\lambda\]
and
\[\widehat{\vol}_{\chi}(\ov{D}) \leq (d + 1)!\int_{\Delta(D)} G_{\ov{D}}d\lambda,\]
with equality if $\inf_{\alpha\in\Delta(D)} G_{\ov{D}}(\alpha) > -\infty$.
\end{theorem}

This in turn allows a comparison with heights and intersection numbers by the (arithmetic) Hilbert-Samuel theorem (see \cite[Proposition 1.31]{debarre_higher_dimensional_alggeo} and \cite[Theorem 5.3.2]{Moriwaki_Adelic_divisors_on_Arithmetic_Varieties}).

\begin{theorem}
    If $D$ is nef, then $\vol(D) = \deg(D)$. If $\ov{D}$ is semipositive, then $\widehat{\vol}_{\chi}(\ov{D}) = h_{\ov{D}}(X)$.
\end{theorem}

\subsection{(Arithmetic) Okounkov bodies of toric varieties}

There are two ways to associate a convex body $\Delta$ and a concave function $\theta$ on $\Delta$ to the datum of a semipositive toric adelic Cartier divisor $\ov{D}$. The first construction is toric by nature, it is given by the Newton polytope and the toric roof function as in \cite{toricheights}. The second construction does not depend on the toric nature(except for the choice of a flag). It is given by the Okounkov body and the Boucksom-Chen concave transform. In the setting of toric bundles, it is important to relate these constructions to apply results on toric varieties.

The equality (up to translation) of the Newton polytope and roof function on the toric side with the Okounkov body and concave transform is mentioned in passing in \cite[Section 4.5]{Boucksom_Chen_Okounkov_bodies_of_filtered_linear_series} using the work of \cite{wn}. The necessary arguments are presented in \cite[Section 5]{torpos} in detail. We prove the equality of the two constructions here as it is not explicitly stated in \cite{torpos}.

Let us first recall the geometric statement. We want to study big toric divisors on a proper toric variety $\toricvar_\Sigma$. We may refine $\Sigma$ in such a way that it defines a smooth projective variety, see \cite[Chapter 11]{Cox_Little_Schneck_Toric}. This does not change the Newton polytope. We may then take prime toric divisors $D_1,\dots,D_t$ such that

\[
X_\bullet:X_\Sigma\supset D_1 \supset D_1 \cap D_2 \supset\dots \supset D_1\cap \dots \cap D_t =\{p\}.
\]
defines a flag. Let $v_i$ denote the primitive generator of the ray corresponding to $D_i$. Then, the $v_i$ form a basis of $N$ and induce an isomorphism $\bbZ^t \cong N$. Its dual basis determines an isomorphism $M\cong \bbZ^t$.

\begin{proposition}[\cite{lazmus} Proposition 6.1]\label{prop:toric_okounkov}
    Let $D$ be a toric divisor not containing any of the $D_1,\dots,D_t$ in its support. Let $\Delta$ be its Newton polytope. Then the Okounkov body of $\rho(\Delta)$ is $\Delta$ under the identification $M\cong \bbZ^t$. In particular, the valuation of the section $\chi^m$ is $m$.
\end{proposition}

By translating $\Delta$, we can always ensure that the conditions in the above proposition are satisfied. We now introduce the toric roof function and prove an arithmetic analogue of the statement above. We apply \cite[Proposition 5.1]{torpos} in order to give an alternative, but equivalent definition of the toric roof functions.

\begin{definition}\label{defi:local_roof_function}
    Let $\ov{D}$ be a toric metrized $\bbR$-Cartier divisor over a field $K$ with absolute value $|\blank|_v$. Let $\Delta$ be the Newton polytope of the underlying geometric Cartier divisor. Then, the local roof function $\theta$ of $\ov{D}$ is the unique continuous function $\theta_{v,\ov{D}}:\Delta \to \bbR$ satisfying the following condition: For $m \in l\Delta$, the toric section $\chi^m \in \Gamma(X,lD)$ satisfies 
    \[
        - \log \|\chi^m\|_{v,\sup} = \theta_{\ov{D}} (m/l).
    \]
    If $\ov{D}$ is a toric adelic $\bbR$-Cartier divisor over a global field $K$, we define its global roof function as a sum over local roof functions
    \[
        \theta_{\ov{D}} (x) = \sum_{v \in M_K} \weightv \theta_{\ov{D},v} (x).
    \]
\end{definition}

The sum is finite since the local roof function is constantly $0$ at all the places for which $\ov{D}$ carries the canonical metric.

\begin{proposition}\label{prop:arith_toric_bc}
    Let $D$ be a toric divisor not containing any of the $D_1,\dots,D_t$ in its support. Let $\Delta$ be its Newton polytope. Let $\ov{D}$ be an adelic Cartier-divisor obtained by endowing $D$ with toric metrics. Identifying the Newton polytope with the Okounkov body as in Proposition \ref{prop:toric_okounkov} the global roof function agrees with the Boucksom-Chen transform.
\end{proposition}

The main technical result we use is the orthogonality of toric sections. We recall the statement for future use.

\begin{theorem}[Corollary 5.4 \cite{torpos}]\label{thm:orthogonality_toric}
    Let $\ov{D}$ be a toric adelic $\bbR$-divisor on $X$ and $s = \sum_{\Delta \cap M}\gamma_m \chi^m \in \Gamma(X,D)$. Then,
    \[
        \|s\|_{\sup} \geq \max_{m \in \Delta\cap M} \|\gamma_m \chi^m\|_{\sup}.
    \]
\end{theorem}

\begin{proof}[Proof of Proposition \ref{prop:arith_toric_bc}]
    We have that the concave transform $G_{\bar{D},X_\bullet}(m/l)$ is bounded from below by $\sum_v \theta_{\bar{D},v} (m/l)$ as $\chi^m$ provides a $\Q$-section with valuation $m/l$ and height $\theta(m/l)$. Let us now prove that $G_{\bar{D},X_\bullet}(m/l) \leq \theta(m/l)$. By definition of the Boucksom-Chen transform, for any $\epsilon >0$, there is a sufficiently large $N$ such that there is a section $s \in \Gamma(ND)$ of valuation $mN/l$ generated by sections of height $<N(G_{\bar{D},X_\bullet}(m/l) + \epsilon)$. So write $s = \sum s_i$. At least one $s_i$ will be of the form $\gamma_m \chi^m + \sum_{m' \neq m} \gamma_{m'} \chi^{m'}$. By the orthogonality of eigenspaces this vector has height $\geq N\theta(m/l)$ finishing the proof.
\end{proof}

We may alternatively conclude the equality of functions using the equality of their integrals provided by \cite[Theorem 5.6]{torpos} and Theorem \ref{thm:arithmetic_volume}.

\subsection{Operations on arithmetic Chow homology}\label{subsec:chow_homology}

Recall that $\sS$ denotes a geometrically irreducible projective curve over a field or the spectrum of the ring of integers in a number field. Let $\mathscr{X}$ be a flat projective scheme over $\sS$. Then, similarly to algebraic geometry Gillet and Soul\'e require regularity of the scheme in order to define an intersection theory. However, it suffices to require the generic fibre $X=\mathscr{X}_K$ to be smooth in order to define arithmetic Chow groups and operational Chern classes associated to hermitian line bundles, see \cite[Section 2.4]{gillet_soule_arithmetic_rr}. The operational perspective is convenient even if one focuses on regular arithmetic varieties as it allows us to pass to fibre products, such as special fibres, that are no longer necessarily regular.

Let us give a basic recollection of arithmetic Chow groups and the intersection with hermitian line bundles. Let $\mathscr{X}$ be a flat projective scheme over $\sS$ with smooth generic fibre. Then, we can define its arithmetic Chow groups $\ArChow_k(\mathscr{X})$ as equivalence classes of pairs $(Z,g_Z)$ of a $k$-cycle $Z$ on $\mathscr{X}$ and Green currents $g_Z$ for $Z_\bbC$ at all archimedean places modulo rational equivalence and the image of $\partial + \ov{\partial}$. Note that there is a well-defined arithmetic degree map $\ArChow_0(\mathscr{X}) \to \bbR$ by the existence of a proper pushforward map, see \cite[Paragraph 2.2.2]{gillet_soule_arithmetic_rr}.

Let $\ov{L}$ be a hermitian line bundle on $\mathscr{X}$. Then, we can describe the action of the first Chern class $\ac_1(\ov{L}): \ArChow_k(\mathscr{X}) \to \ArChow_{k-1}(\mathscr{X})$ explicitly. Let $[(Z,g_Z)] \in \ArChow_k(\mathscr{X})$ and let $s$ be a meromorphic section of $\ov{L}$ on $Z$ and denote by $\omega_{\ov{L}}$ the curvature form of $\ov{L}$. Then, 
\[
    \ac_1(\ov{L})[(Z,g_Z)] = [(\div(s), -\log|s|\delta_{Z(\bbC)} + \omega_{\ov{L}} g_Z)].
\]
This is independent of the choice of section and additive in the tensor product of hermitian line bundles. If $\ov{M}$ is a further hermitian line bundle, the Chern classes commute, i.e.\ $\ac_1(\ov{L})\ac_1(\ov{M})=\ac_1(\ov{M})\ac_1(\ov{L})$ as operations on Chow groups. In particular, there is a multilinear intersection pairing
\begin{align*}
    \aPic(\mathscr{X})^k \times \ArChow_k(\mathscr{X}) &\to \bbR\\
    (\ov{\mathscr{L}}_1,\dots,\ov{\mathscr{L}}_k,[(Z,g_Z)]) &\mapsto \adeg \left(\ac_1(\ov{\mathscr{L}}_1)\cdots \ac_1(\ov{\mathscr{L}}_k)[(Z,g_Z)]\right)
\end{align*}
which is symmetric in the first $k$ entries.

We now define a preliminary arithmetic analogue of operational Chow groups defined by Fulton and MacPherson, see \cite[Chapter 17]{Fulton_Intersection_Theory}. We refrain from considering general bivariant cohomology groups as we are primarily interested in cohomology. This is a preliminary definition since one may for different applications need to require further compatibility relations for instance for proper morphisms that are not generically smooth. However, there is as of now no definition of proper pushforward of arithmetic Chow groups for morphisms that are not generically smooth.

\begin{definition}
    Let $\mathscr{X}$ be a projective scheme over $\sS$ with smooth generic fibre. We will work in the category of projective schemes over $\sS$ with smooth generic fibre.
    
    An element of the $k$-th arithmetic Chow cohomology $A^k(\mathscr{X})$ is defined as a collection of maps $\phi_{\mathscr{Y}}\ArChow_*(\mathscr{Y}) \to \ArChow_{*-k}(\mathscr{Y})$ indexed by maps $\mathscr{Y}\to\mathscr{X}$ (of projective schemes over $\sS$) compatible with proper pushforward, flat pullback and intersection products.
    \begin{enumerate}
        \item For any $h:\mathscr{Y}' \to \mathscr{Y}$ proper map such that $h_K$ is smooth, we have $\phi_{\mathscr{Y}}\circ h_* = h_*\circ\phi_{\mathscr{Y}'}$. For the fundamental class $[\mathscr{Y}']$, this holds even if $h_K$ is not smooth.
        \item For any flat or lci map $h:\mathscr{Y}' \to \mathscr{Y}$, we have $\phi_{\mathscr{Y}'}\circ h^* = h^*\circ\phi_{\mathscr{Y}}$.
    \end{enumerate}
    Here, $\ArChow_*(\mathscr{X})$ denotes the arithmetic Chow group when $\mathscr{X}$ is flat over $\sS$ and the usual Chow group if it maps to a closed point.
\end{definition}

\begin{example}
    The Chern character of a hermitian vector bundle on $\mathscr{X}$ induces elements in the operational Chow groups $A^*(\mathscr{X})$ by \cite[Theorem 4]{gillet_soule_arithmetic_rr}. Furthermore, if $\mathscr{X}$ is regular, there is a map $\ArChow^*(\mathscr{X}) \to A^*(\mathscr{X})$ by \cite[Theorem 3]{gillet_soule_arithmetic_rr}. One can obtain further elements in operational Chow groups by pullback.
\end{example}

\begin{proposition}\label{prop:intersect_integrable_line_bundles_with_cycles}
    The above intersection pairing extends uniquely to a pairing
    \begin{align*}
        \intPic(X)^k \times A^{d-k}(\mathscr{X}) &\to \bbR\\
        (\ov{L}_1,\dots,\ov{L}_k,\phi) &\mapsto \adeg \left(\ov{L}_1\cdots\ov{L}_k\cdot\phi\right)
    \end{align*}
    allowing any integrable line bundle.
\end{proposition}

\begin{proof}
    Suppose first that $\pi:\mathscr{X}'\to\mathscr{X}$ is a further projective model dominating $\mathscr{X}$. Then, we first extend the pairing to $\aPic(\mathscr{X}')$ by setting
    \begin{align*}
    \aPic(\mathscr{X}')^k \times A^{d-k}(\mathscr{X}) &\to \bbR\\
    (\ov{\mathscr{L}}_1,\dots,\ov{\mathscr{L}}_k,\phi) &\mapsto \adeg \left(\ac_1(\ov{\mathscr{L}}_1)\cdots \ac_1(\ov{\mathscr{L}}_k)\phi_{\mathscr{X}'}([\mathscr{X}'])\right).
    \end{align*}

    The definition does not depend on the choice of $\mathscr{X}'$ by the compatibility of $\phi$ with pushforward maps. We are left to show that one can extend the intersection number to integrable line bundles.

    For this assume that $L,L_2,\dots,L_k$ are line bundles on $X$ and $\phi \in A^{d-k}(\mathscr{X})$ is an arithmetic cycle. Suppose that $\ov{\mathscr{L}},\ov{\mathscr{L}}_2,\dots,\ov{\mathscr{L}}_{k}$ are hermitian $\bbQ$-line bundles on $\mathscr{X}'$ whose generic fibres are $L,L_2,\dots,L_k$. Suppose $\ov{\mathscr{L}}_2,\dots,\ov{\mathscr{L}}_{k}$ are relatively nef at a place $v$ and $\ov{\mathscr{L}}$ is of the form $\O(\mathscr{D})$ for an effective Cartier divisor supported on the special fibre at $v$ whose associated Green's function is bounded by $C$. Then, we need to bound
    \[
    \adeg \left(\ac_1(\ov{\mathscr{L}})\ac_1(\ov{\mathscr{L}}_2)\cdots \ac_1(\ov{\mathscr{L}}_k)\phi_{\mathscr{X}'}([\mathscr{X}'])\right)
    \]
    in terms of $C$, $L_2,\dots,L_k$ and $[(Z,g_Z)]$. In the archimedean case this is taken to be meant as $\mathscr{D}$ is given by a non-negative smooth function bounded by $C$. Let $[(Z,g_Z)]$ be a representative of $\phi_{\mathscr{X}'}([\mathscr{X}'])$.
    
    Suppose first that $v$ is archimedean. Note that the curvature form $\omega(g_Z)$ is the difference of strongly positive forms. Locally, this follows from \cite[Section III.1.4]{demailly2012complex}. We may perform a partition of unity to write $\omega(g_Z)$ as the difference of strongly positive forms $\omega^+-\omega^-$. Let $\omega_2,\dots,\omega_k$ be the curvature forms of $\ov{\mathscr{L}}_2,\dots,\ov{\mathscr{L}}_{k}$. Then, by assumption $\omega_2 \cdots \omega_k \omega^+$ and $\omega_2 \cdots \omega_k \omega^-$ are positive measures of mass $\leq K$. Then, 
    \[
    |\adeg \left(\ac_1(\ov{\mathscr{L}})\ac_1(\ov{\mathscr{L}}_2)\cdots \ac_1(\ov{\mathscr{L}}_k)\phi_{\mathscr{X}'}([\mathscr{X}'])\right)| \leq CK.
    \]
    Hence, the intersection product on semipositive line bundles is continuous with respect to supremum norm at archimedean places and hence extends to integrable metrics at archimedean places.
    
    Let us now consider the case of a finite place. Denote the inclusion of the special fibre of $\mathscr{X}'$ by $i:\mathscr{X}'_v \into \mathscr{X}'$. By assumption, $\ac_1(\phi^*\ov{\mathscr{L}})\mathscr{X}'$ is represented by a cycle $\mathscr{D}$ supported on $\mathscr{X}'_v$. It therefore suffices to compute the algebraic intersection $i^*\ac_1(\ov{\mathscr{L}}_2)\cdots i^*\ac_1(\ov{\mathscr{L}}_{k})i^*\phi\cap \mathscr{D}$ on the special fibre by the compatibility with lci morphisms. It suffices to write $i^*\phi$ as the difference of nef dual cycle classes on each irreducible component as nefness is preserved under pullback and under intersection with nef divisors. One can write any dual cycle class as the difference of two nef cycles by the full-dimensionality of the nef cone proven in \cite[Lemma 3.7]{fulger_lehmann_positivity_of_dual_cycle_classes}.
\end{proof}

Instead of taking $\gamma$ in operational cohomology one may use a subclass of $b$-cycles. This has the expected homological flavour. However, we lack understanding of positive cones in this setting.

\begin{definition}
    Let $X$ be a smooth proper variety over $\bbQ$. Let $\sM$ denote the poset of flat proper models of $X$ over $\bbZ$. Then, we define
    \[
        \ExtChow_k(X) \defeq \varprojlim_{\mathscr{X}\in \sM_X}{\ArChow_k(\mathscr{X})},
    \]
    as a cofiltered limit where the transition maps are given by proper pushforward maps. Inside of $\ExtChow_k(X)$ we define a positive cone $\ExtChow^+_k(X)$ as the cofiltered limit of positive cones in $\ArChow_k(\mathscr{X})$. Here, an element $\gamma \in \ArChow_k(\mathscr{X})$ is called positive at a place $v$ if for Cartier divisors $\mathscr{D}_1,\dots,\mathscr{D}_k$ such that $\mathscr{D}_1$ is effective supported on the fibre at $v$ and $\mathscr{D}_2,\dots,\mathscr{D}_k$ are relatively nef, the intersection product satisfies $\mathscr{D}_1\cdots\mathscr{D}_k\cdot \gamma \geq 0$. We call it positive if it is positive at all finite places. We define integrable ($b$-)cycles on $X$ by
    \[
        \ExtChow^{\textrm{int}}_k(X)=\ExtChow^+_k(X)-\ExtChow^+_k(X).
    \]
\end{definition}

\begin{proposition}\label{prop:intersect_integrable_line_bundles_with_cycles}
    The intersection pairing extends uniquely to a pairing
    \begin{align*}
        \intPic(X)^k \times \ExtChow^{\textrm{int}}(X) &\to \bbR\\
        (\ov{L}_1,\dots,\ov{L}_k,[(Z_{\mathscr{X}},g_Z)]_{\mathscr{X}}) &\mapsto \adeg \left(\ov{L}_1\cdots\ov{L}_k\cdot[(Z_{\mathscr{X}},g_Z)]_{\mathscr{X}}\right)
    \end{align*}
    allowing any integrable line bundle.
\end{proposition}

\begin{proof}
    We assume first that all line bundles are endowed with model metrics on some model $\mathscr{X}$. We may compute the intersection number on this model since it does not depend on the model by the projection formula.

    The definition of $\ExtChow^{\textrm{int}}(X)$ is chosen precisely such that the approximation statement can be deduced as before. 
\end{proof}

\begin{remark}
    In fact, an element in $A(\mathscr{X})$ gives rise to an element of $\ExtChow(X)$. The fact that this lies in $\ExtChow^{\textrm{int}}(X)$ is the content of Proposition \ref{prop:intersect_integrable_line_bundles_with_cycles}. There are no pull back maps for $\ExtChow$ in general. Still, the intersection number $\arho(\ov{\Delta})^{t+i}\pi^*\gamma$ in Theorem \ref{thm:intersection_formula} can be defined for $\gamma \in \ExtChow(B)$ since it suffices to compute it on models of the toric bundle that are flat over a model of the base. The proof applies without modification.
\end{remark}

\begin{example}
    Let $Z \subset X$ be a cycle and $g_Z$ be a Green current for $Z$. Then the family of cycles $[(\ov{Z}_{\mathscr{X}},g_Z)]_{\mathscr{X}}$, where $\ov{Z}_{\mathscr{X}}$ denotes the closure of $Z$ inside $\mathscr{X}$, defines an element in $\ExtChow^+_k(X)$. This follows from comparing the intersection number with the height restricted to $Z$, cf.\ \cite[Section 4.5]{soule_arithmetic_intersection_in_book}. 
\end{example}

\section{\label{sec:toric_bundles}Toric bundles}

Analogous to the toric compactification of $\bbT$-torsors in algebraic geometry we would like such a compactification in the arithmetic setting. Despite the lack of a total space, we can apply Arakelov geometry to imitate the algebro-geometric construction where toric Cartier divisors induce line bundles on the total space, cf.\ \cite[Construction 2.3.5]{chambert-loir_tschinkel_arithmetic_torsors}. All torsors will be torsors in the Zariski topology. 

\subsection{Categories of torus bundles}
We can use characters to understand torus bundles more closely. More precisely, let $\catPic(\B)$ be the category of $\bbG_m$-torsors on $\B$ with morphisms given by isomorphisms. It naturally has the structure of a symmetric monoidal category given by the tensor product. Given a $\T$-torsor, every character $\T \to \bbG_m$ gives rise to an element in $\catPic(\B)$. We define the symmetric monoidal category $\sM$ of characters by setting its elements to be $M = \Hom(\T,\bbG_m)$ with no non-trivial morphisms and monoidal structure given by addition. We will call a category associated to finitely generated free abelian groups in this way \emph{lattice category}. Let $\Bun_\T(\B)$ denote the symmetric monoidal category of $\T$-torsors over $\B$\footnote{This notation is reminiscent of the $\B$-valued points of the stack $\Bun_G$ studied among others in the context of the geometric Langlands program as in \cite{beilinson_drinfeld_quantization_hitchin_integrable_system}. While they consider the stack $\Bun_G(X)$ for a curve $X$, we consider the stack $\Bun_G(*)$.}. There is a monoidal equivalence of categories
    \[
    \Bun_\T(\B) \to \Fun^{\otimes}(\sM, \catPic(\B)).
    \]
Here $\Fun^{\otimes}(\sM, \catPic(\B))$ denotes the category of monoidal functors with the monoidal structure given by the tensor product on $\catPic(\B)$ and monoidal natural transformations as morphisms. This equivalence identifies $\T$-torsors with linear maps $M\to \Pic(B)$. This expresses that the category of $\T$-torsors is in equivalence to collections of line bundles $(\Tbun(m))_{m\in M}$ on $B$ indexed by $M$ with compatible identifications $\Tbun(m_1 + m_2) \cong \Tbun(m_1) \otimes \Tbun(m_2)$. The morphisms between such maps are collections of isomorphisms of $\bbG_m$-bundles indexed over $M$ compatible with the tensor product.

We redefine torus bundles to be monoidal functors from a lattice category $\sM$ to $\catPic(\B)$. Over a field $\K$ endowed with an absolute value, we let $\mcatPic(\B)$ denote the category of metrized line bundles. Over a global field, we let $\acatPic(\B)$ denote the category of adelically metrized line bundles.

\begin{definition}
    Let $\B$ be a variety over a field $\K$ endowed with an absolute value. We define a \emph{metrized torus bundle} to be a symmetric monoidal functor $\mTbun: \sM \to \mcatPic(\B)$ from a lattice category $\sM$. If $\K$ is a global field we define an adelic torus bundle as a symmetric monoidal functor $\aTbun: \sM \to \acatPic(\B)$. 
    
    We will extend the meaning of qualifiers from the case of line bundles to torus bundles. We call a metrized/adelic torus bundle integrable if its image consists of integrable line bundles. We call it flat if its image consists of flat line bundles. Recall that a line bundle is flat if both it and its dual are semipositive. We call a torus bundle algebraic if the line bundles in the image are algebraic, i.e.\ the metric is induced by a $\bbQ$-line bundle on a model of $\B$. We say a torus bundle has model metrics if the line bundles in the image have model metrics, i.e.\ the metric is induced by a line bundle on a model of $\B$.
\end{definition}

Torus bundles whose image consists of semipositive line bundles are automatically flat. We can identify isomorphism classes of $\T$-torsors with linear maps $M\to \Pic(B)$. The same principle holds for metrized $\T$-bundles.

\begin{remark}
    We will use adelic notation such as $\aTbun$ as stand-in for the local equivalent $\mTbun$. This does not lead to confusion since all local considerations easily globalize. We will use the word metrized, both in the local context as well as for adelically metrized.
\end{remark}

\begin{definition}
    Let $\sT_1$ and $\sT_2$ be torus bundles on $\B$ for split tori $\T_1$ and $\T_2$ respectively. Let $\phi:\T_1 \to \T_2$ be a group homomorphism. A morphism $f:\Tbun_1 \to \Tbun_2$ over $B$ is $\phi$-equivariant if the diagram
    \begin{center}
    \begin{tikzpicture}[commutative diagram/.style={matrix of math nodes, row sep=1cm, column sep=1cm}]
    \matrix[commutative diagram] (m) {
        \T_1 \times \sT_1 & \T_2 \times \sT_2 \\
        \sT_1 & \sT_2 \\
    };
    \draw[->] (m-1-1) -- node[above] {$\phi \times f$} (m-1-2);
    \draw[->] (m-1-1) -- node[left] {$\mu$} (m-2-1);
    \draw[->] (m-1-2) -- node[right] {$\mu$} (m-2-2);
    \draw[->] (m-2-1) -- node[below] {$f$} (m-2-2);
    \end{tikzpicture}
    \end{center}
    commutes. We say that $f$ is equivariant if it is $\phi$-equivariant for some $\phi$.
\end{definition}

We can now define categories of torus bundles from the functor perspective.

\begin{definition}\label{defi:torus_bundle}
    We define categories of torus bundles over a base $\B$ with the additional flexibility of varying the torus. Morphisms from a torus bundle $\Tbun: \sM \to \catPic(\B)$ to $\Tbun': \sM' \to \catPic(\B)$ are given by a morphism of tori encoded by a monoidal functor $\alpha:\sM'\to\sM$ and an equivariant morphism between the torus bundles encoded by a monoidal natural transformation $F:\Tbun\circ \alpha\to\Tbun'$. This defines a category $\catTbun(B)$. Applying the same approach to the metrized setting yields categories $\mcatTbun(B)$ and $\acatTbun(B)$.
\end{definition}

\begin{proposition}\label{prop:equivariant_morphisms}
    For fixed $\alpha:\sM' \to \sM$, the set of equivariant morphisms is in bijection to the sections of $(\Tbun\circ\alpha)\otimes (\Tbun')^\vee$.
\end{proposition}

The above observation shows that if $\Tbun$ and $\Tbun'$ are metrized and $\Tbun \to \Tbun'$ is a morphism of the underlying torus bundles, there is a function on $\B^{\an}$ giving the norm of the morphism. Only morphisms of constant norm $1$ are considered to be morphisms of metrized torus bundles.

\begin{example}
    For any (metrized) torus bundle $\Tbun$ and integer $n$, there is an $n$-th power map $\Tbun \to \Tbun^{\otimes n}$, where $\Tbun^{\otimes n}$ is defined by $\Tbun^{\otimes n}(m) = \Tbun(m)^{\otimes n}$. It is defined by the map $n:\sM \to \sM$ and the collection of isomorphisms $\Tbun(nm)\to\Tbun(m)^{\otimes n}$ that is inherent in the definition of the monoidal functor $\Tbun$. More generally if $A:M'\to M$ we define the torus bundle $\Tbun^{\otimes A}$ by $\Tbun^{\otimes A}(m') = \Tbun(A \cdot m')$
\end{example}

We would now like to have a notion of morphism of torus bundles that allows for a change in base scheme. Let $f: \B' \to \B$ be a morphism of schemes and $\Tbun$ be a (possibly metrized) torus bundle over $\B$. Then, we may define the pullback $f^*\Tbun = f^* \circ \Tbun$ as the composition of $\Tbun$ and the pullback functor on line bundles, cf.\ \cite[Proposition 1.2.4]{chambert-loir_tschinkel_arithmetic_torsors}.

\begin{definition}
    A morphism of torus bundles $f=(f_t,f_b):(\Tbun_1\to\B_1) \to (\Tbun_2\to\B_2)$ is the datum of a map $f_b:\B_1 \to \B_2$ and an equivariant map $f_t:\Tbun_1 \to f_b^* \Tbun_2$. This defines categories $\catTbun$, $\mcatTbun$ and $\acatTbun$.
\end{definition}

\begin{example}\label{eg:semiabelian}
    This setup provides a suitable setting for the proof of the Weil-Barsotti formula. Let $G$ be a semiabelian variety with a split torus part $\T$ and abelian quotient $A$ with quotient map $\pi$. The addition on $G$ has to be a morphism of toric bundles $f:G\times G \to G$ with underlying map of base spaces $f_b=m_A:A \times A \to A$, the addition on $A$, and map on tori $m_\T: \T \times \T \to \T$. Let $\Tbun$ be a torus bundle over $A$. Then, an equivariant morphism $\Tbun\times\Tbun \to \Tbun$ of the form exists if and only if $m^* \Tbun \cong \Tbun \boxtimes \Tbun$. This is precisely the case when the image of $\Tbun$ consists of algebraically trivial line bundles. The map thus defined is unique, once a rigidification $\T \cong \pi^{-1}(e)$ is chosen. One can easily check that this morphism yields a group structure.
\end{example}

\begin{definition}
    A \emph{dynamical torus bundle} is a tuple $(\B,f,\Tbun,f^*\Tbun \cong \Tbun^{\otimes r})$ consisting of a proper base variety $\B$, an endomorphism $f:\B\to\B$, a torus bundle $\Tbun$ and an isomorphism $f^*\Tbun \cong \Tbun^{\otimes r}$ for some $r>1$. This data gives rise to canonical metrics on $\Tbun$. Over a global field it gives rise to an adelic torus bundle. Its isomorphism class does not depend on the chosen isomorphism.
\end{definition}

\begin{remark}
    One could more generally define a dynamical torus bundle to be a tuple $(\B,f,\Tbun,f^*\Tbun \cong \Tbun^{\otimes A})$ consisting of a proper base variety $\B$, an endomorphism $f:\B\to\B$, a torus bundle $\Tbun$ and an isomorphism $f^*\Tbun \cong \Tbun\circ A$ for some endomorphism $A:M\to M$ which is diagonalizable over $\bbR$ and whose eigenvalues are real numbers $>1$.
\end{remark}

\subsection{Toric bundles and their morphisms}
\begin{definition}
    A \emph{toric bundle} over a scheme $\B$ is a scheme $\toricbun$ over $\B$ with the action of a torus $\T$ such that, Zariski locally on $U \subseteq \B$, $\toricbun$ is $\T$-equivariantly of the form $\toricvar_\Sigma\times U$ for a fixed fan $\Sigma$ in $N_\bbR$.
    
    There is an underlying $\T$-torsor $\Tbun$ such that $\toricbun$ is of the form $\toricbun_\Sigma = (\Tbun \times \toricvar_\Sigma)/\T$, where $x \in \T$ acts by $(x,x^{-1})$. The quotient $(\sT \times \toricvar_\Sigma)/\T$ can be defined using Zariski descent, c.f.\ \cite[Construction 2.1.2]{chambert-loir_tschinkel_arithmetic_torsors}. We will call $\Sigma$ the underlying fan of $\toricbun$.
\end{definition}

\begin{remark}
    It is useful to consider less restrictive definitions in other contexts. For instance the only model of a toric variety over a DVR that is a toric bundle is the canonical model.
\end{remark}

We define a class of maps of toric bundles respecting the toric structure. We first give a definition for a fixed base variety $\B$ inspired by the case of toric varieties.

\begin{definition}
    Let $\toricbun_1$ and $\toricbun_2$ be toric bundles on $\B$ for split tori $\T_1$ and $\T_2$ respectively. Let $\phi:\T_1 \to \T_2$ be a group homomorphism. A morphism $f:\sT_1 \to \sT_2$ over $B$ is $\phi$-equivariant if the diagram
    \begin{center}
    \begin{tikzpicture}[commutative diagram/.style={matrix of math nodes, row sep=1cm, column sep=1cm}]
    \matrix[commutative diagram] (m) {
        \T_1 \times \toricbun_1 & \T_2 \times \toricbun_2 \\
        \toricbun_1 & \toricbun_2 \\
    };
    \draw[->] (m-1-1) -- node[above] {$\phi \times f$} (m-1-2);
    \draw[->] (m-1-1) -- node[left] {$\mu$} (m-2-1);
    \draw[->] (m-1-2) -- node[right] {$\mu$} (m-2-2);
    \draw[->] (m-2-1) -- node[below] {$f$} (m-2-2);
    \end{tikzpicture}
    \end{center}
    commutes. We say that $f$ is equivariant if it is $\phi$-equivariant for some $\phi$. It will be called \emph{non-degenerate} if its fiberwise image intersects the open dense torus on all fibers and \emph{degenerate} otherwise.
\end{definition}

\begin{example}
    Let $\toricvar_\Sigma$ be a toric variety. The $\T$-orbits on $\toricvar_\Sigma$ are in bijection to the cones of $\Sigma$. The closure $V(\sigma)$ of a $\T$-orbit corresponding to a cone $\sigma\in\Sigma$ is itself naturally a toric variety though for a quotient $\T'$ of the original torus $\T$. Its character lattice is identified with $M\cap \sigma^{\perp}$. The fan defining $V(\sigma)$ is denoted $\Sigma(\sigma)$. Given a toric bundle of the form $\toricbun_\Sigma$, there is a closed subvariety $\calV(\sigma)$ given by $(\Tbun \times \calV(\sigma))/\T$. This is again a toric bundle. The isomorphism class of its underlying torus bundle is given by $M\cap \sigma^{\perp} \to M \to \Pic(\B)$.

    A morphism of this form is called a closed embedding of toric bundles.
\end{example}

Let us classify non-degenerate equivariant morphisms of toric bundles using combinatorial data. We can apply the theory of torus bundles since a morphism of toric bundles $\toricbun_1 \to \toricbun_2$ restricts to a map of the underlying torus bundle $\Tbun_1 \to \Tbun_2$. Note that a toric bundle is defined by the data of a torus bundle together with a fan $\Sigma$ on $N_\bbR$.

\begin{proposition}
\label{prop:equivariant_morphisms}
    Let $\toricbun_1$ and $\toricbun_2$ be toric bundles defined by the data of underlying torus bundles $\Tbun_i:\sM_i \to \catPic(\B)$ and fans $\Sigma_i$. Let $\phi:\T_1 \to \T_2$ be a group homomorphism such that for every cone $\sigma_1 \in \Sigma_1$ there exists a cone $\sigma_2 \in \Sigma_2$ such that $H(\phi)(\sigma_1) \subseteq \sigma_2$. Then the non-degenerate $\phi$-equivariant morphisms from $\toricbun_1$ to $\toricbun_2$ are in bijection with the $\phi$-equivariant maps of torus bundles $\Tbun_1 \to \Tbun_2$. If the condition on the fan is not satisfied there are no $\phi$-equivariant morphisms.
    
    In particular, $\phi$-equivariant maps are in bijection to sections of the $\T_2$-torsor defined by
    \[
    \Tbun_2^\vee \otimes (\Tbun_1 \circ H^\vee(\phi)):\sM_2 \to \catPic(\B).
    \]
\end{proposition}

\begin{proof}
    This is a combination of Proposition \ref{prop:equivariant_morphisms} and \cite[Theorem 3.2.4]{toricheights}.
\end{proof}

\begin{example}
    Degenerate equivariant morphisms are not as easy to describe.  For instance, there is no description of equivariant morphisms as the composition of non-degenerate morphisms and the inclusion of a closed orbit as in the non-relative setting.
    
    An example of a degenerate morphism, where this property fails can be given as follows. Pick as base variety a toric variety $\B$ for the torus $\T$. We consider the trivial toric bundles $\T \times \B$ and $\B \times \B$. Then, $(m,\id):\T \times \B \to \B \times \B$ defines a degenerate morphism if $\B\neq\T$. More precisely, the image of $(m,\id)$ restricted to fibre over $b\in \B$ lands in the orbit of $b$.

    Another property that is not inherited from the absolute case where $\B = *$ is that in order to specify an equivariant morphism $\toricvar_{\Sigma_1}\to\toricvar_{\Sigma_2}$ it suffices to give a map of tori $\phi:\T_1 \to \T_2$ respecting the fans and a $\phi$-equivariant map $\Tbun_1 \to \toricvar_{\Sigma_2}$. The failure of this property is illustrated by the fact that the map $(m,\id):\T \times \B \to \B \times \B$ does not extend to a map $\B \times \B \to \B \times \B$ for $\B \neq \T$.
\end{example}

The above example motivates the following definition.

\begin{definition}
    An equivariant morphism is called a \emph{morphism of toric bundles} if it is the composition of a non-degenerate equivariant morphism followed by a closed embedding of toric bundles.
\end{definition}

We would like to allow for changes of the base variety. For this note that the base change of a toric bundle is the toric bundle associated to the same fan $\Sigma$ and the pullback torus bundle. We will use also use the word pullback and related notation in the context of toric bundles.

\begin{definition}
    A morphism of toric bundles $f=(f_t,f_b):(\toricbun_1\to\B_1) \to (\toricbun_2\to\B_2)$ is the datum of a map $f_b:\B_1 \to \B_2$ and a morphism of toric bundles $f_t:\toricbun_1 \to f_b^* \toricbun_2$.
\end{definition}

\begin{example}\label{eg:nth_power_map}
    Let $\Tbun$ be a torus bundle and $\Sigma$ a fan on the co-character space of its torus. Let $\toricbun$ be the toric bundle associated to this data and denote by $\toricbun^{\otimes n}$ the toric bundle associated to $\Tbun^{\otimes n}$ and $\Sigma$. Then, the map of torus bundles $[n]:\Tbun\to\Tbun^{\otimes n}$ extends to a non-degenerate map of toric bundles $[n]:\toricbun\to\toricbun^{\otimes n}$.
\end{example}

\begin{example}
    Let us extend Example \ref{eg:semiabelian}. From Proposition \ref{prop:equivariant_morphisms} it easily follows that the multiplication on $G$ extends to actions on its compactifications as toric bundles. We note that multiplication by $n$ on semiabelian varieties can be seen as the unique toric morphism with underlying map $f_b=[n]:A\to A$ and map on tori $[n]:\T \to \T$ respecting the rigidification. Again we may apply Theorem \ref{prop:equivariant_morphisms} to obtain that it extends to any compactification as a morphism of toric bundles. We remark that homomorphisms of semiabelian varieties can be viewed as morphisms of toric bundles and that they extend to morphisms of compactifications if this is true on the toric part.
\end{example}

\subsection{Divisors on toric bundles}

Although we are mainly interested in the case of torus invariant Cartier divisors on toric bundles it does not add much complexity to consider more general $\T$-linearized sheaves. They are convenient as contrary to Cartier divisors one can always pull them back. 

\begin{definition}
    A $\T$-linearized quasicoherent sheaf $\sF$ on a toric variety $\toricvar$ consists of a quasicoherent sheaf $\sH$ on $\toricvar$ together with an isomorphism $m^* \sH \cong pr_2^* \sH$, where $m$ denotes the action $\T \times \toricvar \to \toricvar$. The category of $\T$-linearized sheaves is denoted by $\QCoh_\T(\toricvar)$.
\end{definition}

A toric Cartier divisor $D$ gives rise to a $\T$-linearized line bundle $\O(D)$ with the $\T$-linearization given by the morphism $m^*\O(D)\to pr_2^* \O(D)$ being given by sending $m^* s_D$ to $pr_2^* s_D$ where $s_D$ is the distinguished rational section. If a toric Cartier divisor additionally is endowed with a toric metric, it follows that the isomorphism $m^*\O(D)\to pr_2^* \O(D)$ is in fact an isometry over the preimage of $0\times X^{\an} \subseteq N_\bbR \times X^{\an}$ in $(\T\times X)^{\an}$.

We fix the following setup. Let $\toricvar$ be a variety with an action of $\T$ and let $\Tbun$ be a $\T$-bundle on $\B$. Let $\toricbun = (\Tbun \times \toricvar)/\T$. We call $\toricbun \to \B$ a $\T$-fibre bundles. A morphism of torus fibre bundles is a map induced by a map $\Tbun_1 \to \Tbun_2$ of torus bundles and an equivariant map of varieties $X \to Y$ with torus actions. We obtain the notion of (adelically) metrized torus fibre bundles by adding metrics to the underlying torus bundle.

\begin{proposition}
    There is a unique functor
    \[
        \rho:\QCoh_\bbT(\toricvar) \to \QCoh(\toricbun)
    \]
    compatible with tensor products satisfying the following conditions:
    \begin{enumerate}
        \item If $B = \{*\}$ and $\Tbun$ is the trivial bundle, the map is the identity.
        \item It commutes with pullback for morphisms of torus fibre bundles, i.e.\ if $\sF:\toricbun_1 \to \toricbun_2$ is a morphism with underlying morphism of varieties with torus action $f:\toricvar_1 \to \toricvar_2$, then $\sF^*\rho(\sH) = \rho(f^* \sH)$.
    \end{enumerate}
\end{proposition}

\begin{proof}
    Locally on a Zariski open $U \subseteq \B$ we may choose a  trivialization of $\Tbun$. This may be seen as a collection of sections $s_m \in \Gamma(\Tbun(m))$ satisfying $s_{m_1 + m_2} = s_{m_1} \otimes s_{m_2}$. Since Cartier divisors are a local notion we may assume w.l.o.g.\ that $U = \B$. The trivialization of $\Tbun$ induces a trivialization $s:\toricbun \to \toricvar_\Sigma$. In this situation we set $\rho$ to be $s^*$. The ambiguity in this definition is resolved by applying the $\T$-linearization. The constructed map is the unique map satisfying the imposed conditions.
\end{proof}

\begin{remark}
    One may define $\rho:\Div_\bbT(\toricvar_\Sigma) \to \Div_\T(\toricbun_\Sigma)$ on the level of Cartier divisors. The divisors $D$ and $D'$ being linearly equivalent does not imply that $\rho(D)$ and $\rho(D')$ are linearly equivalent.
\end{remark}

In the metrized setting it is natural to replace $\T$ by the elements of unit norm and $\Tbun$ by the part of the torus bundle of unit norm. This is dictated to us since we want compatibility of the two constructions in the case of model metrics. We will avoid the use of quotients due to their technical difficulty in algebraic and nonarchimedean geometry.

We denote by $\Vect_\T(\toricvar)$ the category of $\T$-linearized vector bundles on $\toricvar$. We denote by $\mVect_\T(\toricvar)$ the category of metrized vector bundles with a $\T$-linearization which is an isometry over the preimage of $0\times X^{\an}\subseteq N_\bbR \times X^{\an}$. Let $\aVect_\T(\toricvar)$ be the category of $\T$-linearized adelic vector bundles, i.e.\ the category of $\T$-linearized vector bundles $V$ with metrics at all places such that the $\T$-linearization is an isometry over the preimage of $0\times X^{\an}$ and there is a dense open $U\subseteq \sS$ such that the metrics over $U$ are defined by a model $V_U$.

\begin{proposition}\label{prop:mrho}
    Let $\B$ be a variety over a valued field $\K$ and $\toricbun$ be a metrized toric bundle with fan $\Sigma$ and underlying metrized torus bundle $\mTbun$ over $\B$. Let $\toricvar_\Sigma$ be the toric variety associated to $\Sigma$. Then, there is a unique monoidal functor
    \[
        \mrho:\mVect_\bbT(\toricvar_\Sigma) \to \mVect_\bbT(\toricbun_\Sigma)
    \]
    satisfying the following conditions:
    \begin{enumerate}
        \item If $B = \{*\}$ and $\mTbun$ is the trivial bundle, the map is the identity.
        \item It commutes with pullback for morphisms of metrized torus fibre bundles, i.e.\ if $\sF:\toricbun_1 \to \toricbun_2$ is a morphism with underlying morphism of varieties with torus action $f:\toricvar_1 \to \toricvar_2$, then $\sF^*\mrho(\sH) = \mrho(f^* \sH)$.
        \item Let $\phi:\toricbun_1 \to \toricbun_2$ be a morphism of metrized torus fibre bundles. Then, there is a homomorphism $|f|:M_2 \to \Cont(\B^{\an},\bbR)$ associated to the induced map $f:\Tbun_1 \to \Tbun_2$ giving the norm of $f$. Then, it holds that $f^*\circ \mrho_{\aTbun_2(-|f|)} = \mrho_{\aTbun}\circ\phi^*$. The subscript emphasizes the dependence on the metric of the underlying torus bundle.
    \end{enumerate}
The homomorphism $\mrho$ automatically restricts to $\rho$ on underlying vector bundles.
\end{proposition}

\begin{proof}
We cannot copy the proof from the non-metrized verbatim since we cannot assume the existence of local trivializing sections of constant norm $1$. We can, however, trivialize the underlying torus bundle and correct for the non-triviality of the norm. 

Let $s$ be a local trivialization of $\Tbun$. Then, there is a continuous map $|s|:\B^{an} \to \T^{trop}$ given by the metrics. This allows us to write $\toricbun \cong \B \times \toricvar_\Sigma$. We define the vector bundle by pullback. We proceed to define the metric.

Under the isomorphism $m^* V \simeq pr^* V$ for $m,pr:\T \times \toricvar_\Sigma\to \toricvar_\Sigma$ we have the following relationship on norms. Suppose $v$ is a nowhere vanishing section of $m^* V \simeq pr^* V$ over some open $U\subset \T \times \toricvar_\Sigma$. Then, the quotient of norms $\delta_v=\frac{|v|_{m^* V}}{|v|_{pr^* V}}$ factors over $\T^{trop} \times \toricvar_\Sigma$ by the assumption that the linearization be an isometry over $0\in \T^{trop}$.

Let $v \in V(U)$ for $U \subseteq \toricvar_\Sigma$ be a nowhere vanishing section. We define the norm of $pr^* v \in \rho(V)(\B \times U)$ to be $|v(u)|\cdot \delta_v(|s|,u)$. It is easy to see that it doesn't depend on the choice of section. This construction defines the unique functor satisfying the conditions on $\mrho$.
\end{proof}

\begin{proposition}\label{prop:divisors_on_toricbun}
    Let $\B$ be a variety over a global field $\K$ and $\toricbun$ be an adelic toric bundle with fan $\Sigma$ and underlying adelic torus bundle $\aTbun$ over $\B$. Let $\toricvar_\Sigma$ be the toric variety associated to $\Sigma$. Then, there is a unique homomorphism
    \[
    \arho:\aVect_\bbT(\toricvar_\Sigma) \to \aVect(\toricbun_\Sigma)
    \]
    which restricts to $\mrho$ at each place.
\end{proposition}

\begin{proof}
    We need to show metrics defined by $\mrho$ glue together to yield an adelic metric. Let $\ov{V}$ be an adelic vector bundle. Let $U\subset \sS$ be an open dense subset on which both $\aTbun$ and $\ov{V}$ have a model. By assumption, there is a model $\sB$ over $U$ and a $\T$-torsor $\aTbun_U$ defining the metrics of $\aTbun$ at almost all places. Then, $\rho_{\aTbun_U}(\ov{V}_U)$ defines the metrics at all places included in $U$. Thus one obtains an adelic vector bundle.
\end{proof}

We note that one may analogously define homomorphisms for Cartier divisors $\rho:\Div_\T(\toricvar_\Sigma) \to \Div_\T(\toricvar_\Sigma)$, $\mrho:\mDiv_\T(\toricvar_\Sigma) \to \mDiv_\T(\toricvar_\Sigma)$ and $\arho:\aDiv_\T(\toricvar_\Sigma) \to \aDiv_\T(\toricvar_\Sigma)$. The rest of the section is devoted to study their positivity properties.

\begin{lemma}\label{lemm:semipositive_rho}
    Let $\ov{D}$ be a semipositive toric metrized divisor on $\toricvar_\Sigma$ with underlying divisor $D$ whose Newton polytope is $\Delta$. Let $\ov{L}$ be a metrized line bundle on $\B$ such that $\pi^*\ov{L} + \mTbun(m)$ is semipositive for all $m \in \Delta\cap M_\bbQ$. Then, the line bundle $\pi^*\ov{L} + \mrho(\ov{D})$ is semipositive.
\end{lemma}

\begin{proof}
    Denote the local roof function of $\ov{D}$ by $\theta$ and let $\ov{L}(m) \defeq \pi^*\ov{L} + \mTbun(m) + \theta(m)$. For $m \in \Delta \cap M_\bbQ$, the line bundle $\aTbun(m)+\theta(m)$ on $\B$ parametrizes $\chi^m$-equivariant sections of $\mrho(\ov{D})$. We approximate the metric on $\pi^*\ov{L} + \mrho(\ov{D})$ by the maximum over the semipositive metrics defined on the $\chi^m$-invariant $\bbQ$-sections $\ov{L}(m)$. The semipositivity of $\pi^*\ov{L} +\mrho(\ov{D})$ follows immediately.

    Recall that given a globally generated line bundle $L$ on $X$, a norm $\|\blank\|$ on $H^0(X,L)$ induces a metric on $L$ by the quotient norm under the restriction map $H^0(X,L) \to H^0(x,L^{an}|_x)$ for $x \in X^{an}$. For two metrics on a line bundle we define their distance to be the supremum norm of the function
    \[
    \log \frac{|\blank|'}{|\blank|}.
    \]
    Let $l>1$ be a number such that $\log|l\Delta\cap M|/l < \epsilon$ and the metric induced by the supremum norm on $\Gamma(l D)$ differs from the norm on $\ov{D}$ by at most $\epsilon$. Such an $l$ exists by the characterization of semipositive metrics on semiample line bundles in \cite[Section 3.3]{Chen_Moriwaki_semipositive_extensions} in the non-archimedean and \cite{moriwaki_semiample_semipositive_metrics_archimedean} in the archimedean case. On $\Gamma(l D)$ we know the following inequality of norms from \cite[Corollary 5.4]{torpos}. For $s = \sum_{m \in \frac{1}{l}\Delta \cap M} \gamma_m\chi^m\in \Gamma(l D)$, we have
    \[
    \max_{m \in \frac{1}{l}\Delta \cap M} \|\gamma_m\chi^m\|_{\sup}\leq \|s\|_{\sup} \leq |l\Delta\cap M| \max_{m \in \frac{1}{l}\Delta \cap M} \|\gamma_m\chi^m\|_{\sup}.
    \]
    The expression $\max_{m \in \frac{1}{l}\Delta \cap M} \|\gamma_m\chi^m\|_{\sup}$ defines a norm on $\Gamma(l D)$ and by taking the quotient norm a norm on $D$. The distance of this norm to that on $\ov{D}$ is at most $2\epsilon$.
    
    We now define a norm on $\rho(D) + \pi^* L$ approximating the norm on $\arho(\ov{D}) + \pi^*\ov{L}$. Locally on $B$, we can decompose sections of $l(\rho(D) + \pi^* L)$ into eigenspaces, i.e.\ sections will be of the form $\sum_{m \in \frac{1}{l}\Delta \cap M} \gamma_m$ for sections $\gamma_m \in \Gamma(U,lL(m))$. We metrize $l(\rho(D) + \pi^* L)$ by taking the maximum of the norms on $lL(m)$. The induced metric on $\rho(D) + \pi^* L$ differs at most by $2 \epsilon$ from the metric on $\pi^*\ov{L} +\mrho(\ov{D})$. Furthermore it is semipositive since it is the maximum of semipositive metrics.
\end{proof}

\begin{corollary}
    If $\mTbun$ is integrable, then $\mrho$ preserves integrability. If $\mTbun$ is flat, then $\mrho$ preserves semipositivity.
\end{corollary}

\begin{proof}
    Let $\ov{D}$ be a semipositive toric divisor. We need to show that there exists a semipositive line bundle $\ov{L}$ on the base $B$ such that $\arho(\ov{D}) +\pi^* \ov{L}$ is semipositive. For this it suffices to find a semipositive $\ov{L}$ such that $\aTbun(m)+ \pi^*\ov{L}$ is semipositive for all $m \in \Delta$. We can replace the condition for all $m \in \Delta$ by a condition for finitely many points since semipositive divisors form a convex cone. For any single $m$, this is possible by the assumption that $\aTbun$ be integrable. Summing up over semipositive line bundles yields a single $\ov{L}$ for all $m \in \Delta$. The second claim is immediate from Lemma \ref{lemm:semipositive_rho}.
\end{proof}

\begin{example}
    We freely use the notation of Example \ref{eg:nth_power_map}. Let $(\B,f,\Tbun,f^*\Tbun \cong \Tbun^{\otimes r})$ be a dynamical torus bundle and let $\aTbun$ be $\Tbun$ endowed with its canonical metric. Then, for any fan $\Sigma$ there is an endomorphism of the toric bundle associated to $\Tbun$ and $\Sigma$ given by
    \[
        \sF:\toricbun \xrightarrow{\Box^r} \toricbun^{\otimes r} \cong f^*\toricbun \to \toricbun.
    \]
    There is a canonical isomorphism $\sF^*\rho(D) \cong r\rho(D)$ that induces a canonical metric on $\rho(D)$ which we will denote by $\ov{\rho(D)}$. There is an isometry $\ov{\rho(D)} \cong \arho_{\aTbun}(D)$. This, in particular, applies to multiplication by $r$ on semiabelian varieties.
\end{example}

\subsection{Global sections}
We would now like to study the global sections of vector bundles constructed in the previous section. Let $\toricvar$ be a variety with an action of $\T$ and let $\Tbun$ be a $\T$-bundle on $\B$. Let $\toricbun = (\Tbun \times \toricvar)/\T$ be the corresponding torus fibre bundle. Let $V\in\Vect_\T(\toricvar)$ be a $\T$-linearized vector bundle. Note that the pushforward $\pi_* V$ along a $\T$-equivariant morphism is naturally endowed with a $\T$-linearization. After recalling \cite[Proposition 1.8]{knop_lange_commutative_algebraic_groups_and_intersections_of_quadrics} in our notation, we prove an adaptation to the metrized setting.

\begin{lemma}\label{prop:geometric_global_sections}
    Let $\sF$ be a $\T$-linearized quasi-coherent sheaf on $\toricvar$. Let $\pi:X\to Y$ be a $\T$-equivariant morphism of $\T$-varieties. Let $\Tbun$ be a $\T$-bundle on $\B$ and $\Pi:\sX \to \sY$ be the induced map on torus fibre bundles. Then,
    \[
        \Pi_* \rho(\sF) \cong \rho(\pi_*\sF).
    \]
    In particular, if $X$ is a toric variety and $Y=*$ the pushforward of $\sO(\rho(\Delta))$ can be decomposed into eigenspaces as
    \[
        \pi_* \sO(\rho(\Delta))\cong \bigoplus_{m \in \Delta \cap M} \Tbun(m).
    \]
\end{lemma}

\begin{proof}
    Assume first that $\Tbun$ is trivial. In this case the assertion follows from flat base change. The identification given by flat base change does not depend on the choice of trivialization of $\Tbun$. Hence, the isomorphism globalizes.
\end{proof}

Suppose that $Y=*$ and $\ov{V}$ is a metrized vector bundle on a proper $X$. Suppose further that $\Tbun$ is given a metric structure. Then, the $\pi_* V = H^0(X,V)$ has the structure of a $\T$-linearized metrized vector bundle with norm given by the supremum norm. On the other hand, $\Pi_*\rho(V)$ carries a metric given by the fibrewise supremum norm.

\begin{proposition}\label{prop:metrized_pushforward}
    The isomorphism
        \[
            \Pi_* \mrho(\ov{V}) \cong \mrho(H^0(X,V)_{\sup}).
        \]
    is an isometry.
\end{proposition}

\begin{proof}
    Suppose that $\Tbun$ is trivial and the trivialization is a normed trivialization over $b \in \B^{\an}$. Then by construction, the metrized vector bundle $\mrho(\ov{V})$ restricted to $\Pi^{-1}(b)$ is precisely $\ov{V}$. Hereby the claim follows.
\end{proof}

\begin{lemma}\label{prop:metrized_global_sections}
    The eigenspaces $\Tbun(m)$ of the metrized vector bundle $\pi_* \sO(\mrho(\ov{\Delta}))$ are orthogonal to one another. Each eigenspace $\Tbun(m)$ carries the metric $\mTbun(m) + \theta_v(m)$.
\end{lemma}

\begin{proof}
    The claim follows by Proposition \ref{prop:metrized_pushforward} and a study of global sections on toric varieties. The orthogonality claim follows immediately from Theorem \ref{thm:orthogonality_toric}. The claim on the metrics on each eigenspace follows by the definition of the local roof function in Definition \ref{defi:local_roof_function}.
\end{proof}

\subsection{Arithmetic intersections}
This section is devoted to a generalization of \cite[Proposition 4.1]{cl2}. We work on an integrable smooth projective toric bundle $\toricbun$ with underlying adelic torus bundle $\aTbun \to \B$ with $\B$ smooth projective.

Let $\widetilde{\Sigma}$ be a smooth projective adelic fan with recession fan $\Sigma$. Let $\tau_i$ denote the rays of $\widetilde{\Sigma}$ and let $h_i$ denote the dual polytopes. For notation on adelic polytopes, we refer to Section \ref{subsec:toric_varieties_prelim}.

\begin{proposition}
    \label{empty_intersection}
    Let $\tau_1,\dots,\tau_r$ be rays in $\widetilde{\Sigma}$ not spanning a cone. Let $\toricbun_\Sigma$ be a toric bundle whose underlying torus bundle is endowed with an adelic metric $\aTbun$ with model metrics at finite places and a smooth metric at archimedean places. Let $\gamma \in \ArChow_{r}(\atoricbun)$ be an arithmetic cycle on a model $\atoricbun$ of $\toricbun_\Sigma$. Then,
    \[
        \adeg(\arho(h_1)\cdots\arho(h_r)\gamma) = 0.
    \]
\end{proposition}
\begin{proof}
    Note first that if $\tau_1,\dots,\tau_r$ are not supported at one place the conclusion holds since one may find a regular scheme $\sY$ on which the $\arho(h_i)$ can be defined in such a way that they are still supported at the same places. On $\sY$ the intersection of the $\arho(h_i)$ is empty. Then, the conclusion follows from \cite[Theorem 3 (2)]{gillet_soule_arithmetic_rr}.

    Now consider the case that all $\tau_i$ are supported at a non-archimedean place $v$. Observe that the adelic torus bundle is induced by the pullback of an adelic torus bundle on a regular projective scheme $\sZ$ by projectivity of $B$. Note that the fan $\widetilde{\Sigma}_v \subset N_\bbR \times \bbR_{\geq 0}$ not only defines a model, but also a smooth toric variety of dimension $t+1$. Denote the monomial corresponding to the dual of $\bbR_{\geq 0}$ by $T$. The model $\toricmod_{\widetilde{\Sigma}_v}$ over $R_v$ is given as the closed subscheme of $X_{\widetilde{\Sigma}_v,R_v}$ cut out by $T=\pi$.
    
    The torus bundle $\aTbun \oplus \bbG_m$ on $\sZ$ and the fan $\widetilde{\Sigma}$ define a toric bundle over $\sB$. Since a smooth scheme over a regular base is regular the toric bundle $\toricbun_{\widetilde{\Sigma}_v}$ is in fact regular and the Cartier divisors $\arho(h_i)$ are defined by pullback of $\rho(h_i)$ from there. The intersection of the $\rho(h_i)$ in $\toricbun_{\widetilde{\Sigma}_v}$ is empty. Hence, the claim follows again by \cite[Theorem 3 (2)]{gillet_soule_arithmetic_rr}.

    It remains to deal with the case that the $\tau_i$ are supported at an archimedean place. We allow ourselves to work on the toric bundle $\toricbun_{\Sigma}$.

    Let us prove the vanishing of the current $g_{\arho(h_1)}\omega(\arho(h_2))\cdots\omega(\arho(h_r))$. Note that the Green's functions of the $h_i$ are given by piecewise linear functions $g_i:N_\bbR \to \bbR$. On the preimage of an open subset $U\subset N_\bbR$ where $g_i$ vanishes the metric is flat. In fact, the standard section of $\rho(D_i)$ is flat on this locus. The empty intersection condition in the combinatorial setting implies that $\prod g_i = 0$. Hence, $g_{\arho(\ov{D}_0)}$ vanishes on the support of $\omega(\arho(\ov{D}_1))\cdots\omega(\arho(\ov{D}_t))$. We note that by approximation the intersection number is precisely the integral of $g_{\arho(h_1)}\omega(\arho(h_2))\cdots\omega(\arho(h_r))$ against the curvature form of $\gamma$. This finishes the proof.
\end{proof}

For the $n$-th power map $\Box^n:\toricbun \to \toricbun^{\otimes n}$ introduced in Example \ref{eg:nth_power_map} we have $(\Box^n)^*\arho_{\toricbun^{\otimes n}}(\ov{D}) = \arho_\toricbun([n]^*\ov{D})$. Moreover, for any toric $\bbQ$-Cartier divisor $\ov{D}$ on a model $\toricvar_{\widetilde{\Sigma}}$, there is a unique $\bbQ$-Cartier divisor $\ov{D}'$ on the same model such that $[n]^*\ov{D}' = \ov{D}$. This implies that the above proposition applies more generally to algebraic torus bundles because $(\Box^n)^*\arho_{\toricbun^{\otimes n}}(\ov{D})$ is algebraically metrized. Hence the conclusion stays valid for algebraically metrized torus bundles by passing to $\toricbun^{\otimes n}$.

We need an approximation result in order to apply this for integrable torus bundles that are not necessarily algebraic.

\begin{proposition}
    \label{empty_intersection}
    Let $\tau_1,\dots,\tau_r$ be rays in $\widetilde{\Sigma}$ not spanning a cone. Let $\toricbun_\Sigma$ be a toric bundle whose underlying torus bundle is endowed with integrable adelic metrics $\aTbun$. Let $\gamma \in \ArChow_{r}(\atoricbun)$ be an arithmetic cycle on a model $\atoricbun$ of $\toricbun_\Sigma$. Then,
    \[
        \adeg(\arho(h_1)\cdots\arho(h_r)\gamma) = 0.
    \]
\end{proposition}

\begin{proof}
    
    We approximate $\aTbun$ by algebraic metrics $\aTbun_k$. Pick a basis $m_1,\dots,m_t$ of $M$. For each $m_i$ we can write $\aTbun(m_i)$ as the difference of limits algebraic semipositive metrics, i.e.\ there are semipositive algebraic line bundles $L^+_{k,i}$ and $L^-_{k,i}$ such that for $k\to\infty$ the metrics converge uniformly to semipositive line bundles $L^+_i$ and $L^-_i$ such that $\aTbun(m_i) = L^+_i - L^-_i$. Extending the map $m_i\mapsto L^+_{k,i} - L^-_{k,i}$ by linearity to $M$ allows us to approximate $\aTbun$ by algebraic metrics $\aTbun_k$.
    
    Note that for $C$ big enough the line bundle $C\sum_i\left( L^+_{k,i} + L^-_{k,i}\right) + \arho_{\aTbun_k}(\ov{D}_j)$ is semipositive for all $j=0,\dots,t$ and all $k$. We are done once we show that $\arho_{\aTbun_k}(\ov{D})$ converges uniformly to $\arho_{\aTbun}(\ov{D})$ for $\ov{D} = \ov{D}_j$ for some $j$. Fix a norm $\|\blank\|$ on $N_\bbR$. Since the Green's function $g$ of $\ov{D}$ on $N_\bbR$ is piecewise linear on finitely many polyhedra it is Lipschitz continuous with some Lipschitz constant $K$. Varying the metric on the torus bundle by vector $v(b)$ in $N_\bbR$ of length at most $\epsilon$ at a $b \in \B^{an}$ translates the Green's function by $v(b)$ on the fibre at $b$. By the Lipschitz continuity, we obtain $\| g(x) - g_j(x-v)\|_{\sup} \leq K\|v\|$. The uniform convergence of line bundles $\aTbun_k(m) \to \aTbun(m)$ implies that the difference of $\aTbun_k$ and $\aTbun$ given by $v_k:\B^{an} \to N_\bbR$ converges uniformly to 0.
\end{proof}

\subsection{Successive minima of toric bundles}~\label{sec:suc_min}

Let $X$ be a proper variety and let $\ov{L}$ be an adelic line bundle on $X$ whose underlying line bundle is big. For a real number $\lambda$, the set $X_{\ov{L}}(\lambda)$ denotes the Zariski closure of
\[
    \{ x \in X(\Kbar)\mid  h_{\ov{\sL}}(x) \leq \lambda\}.
\]
Since $X$ is a noetherian topological space the filtration given by varying $\lambda$ has only finitely many stages. It yields the so-called height filtration
\[
    X_0 = \emptyset \subsetneq X_1 \subsetneq \dots \subsetneq X_r = X.
\]
Often, one considers the sub-filtration consisting of the $X_i$ such that $\dim X_i > \dim X_{i-1}$. This filtration is closely related to the notion of Zhang minima. Let $\toricbun$ be a proper toric bundle and $\pi^* \ov{\sM} + \arho(\ov{D})$ be an adelic line bundle on $\toricbun$ whose underlying line bundle is big. Then, there are other natural filtrations.

\begin{lemma}
    The set $\toricbun_\lambda = \ov{\{ x \in \toricbun(\Kbar)\mid  h_{\ov{\sL}}(x) \leq \lambda\}}$ is invariant under the torus action.
\end{lemma}
\begin{proof}
    It is enough to prove that $\toricbun_\lambda(\Kbar)$ is stable under the action of $\T(\Kbar)$ as reduced schemes are determined by their $\Kbar$-points. The height is invariant under the action of torsion points of $\T$. This is checked fibrewise. Torsion points of $\T$ tropicalize to $0$ at all places. Hence, their action does not affect the tropicalization of points in the fibre $\toricvar_\Sigma$ and thus their height.
    
    Since torsion points are Zariski dense in $\T$, it follows that $\toricbun_\lambda$ contains the $\T(\Kbar)$-orbit of any $x \in \toricbun(\Kbar)$. In particular, we may view $\toricbun_\lambda(\Kbar)$ as the closure of the $\T(\Kbar)$-invariant set $\T(\Kbar)\{ x \in \toricbun(\Kbar)\mid  h_{\ov{\sL}}(x) \leq \lambda\}$. The claim now follows from the fact that the action map $\T(\Kbar)\times \toricbun(\Kbar) \to \toricbun(\Kbar)$ is continuous.
\end{proof}

It is natural to consider $\Tbun_\lambda = \ov{\{ x \in \Tbun(\Kbar)\mid  h_{\ov{\sL}}(x) \leq \lambda\}}$ in $\Tbun$, where $\Tbun$ denotes the underlying torus bundle. Then, the resulting filtration
\[
    \Tbun_0 = \emptyset \subsetneq \Tbun_1 \subsetneq \dots \subsetneq \Tbun_r = \Tbun
\]
is the pullback of a filtration on $\B$ which we call the toric filtration. The following equality holds
\[
    \Tbun_\lambda = \pi^{-1}( \ov{\{ b \in \B(\Kbar)\mid \exists x \in \pi^{-1}(b), h_{\ov{\sL}}(x) \leq \lambda\}}).
\]

One can recover the height filtration on $\toricbun$ from filtrations on $\B$ associated to each cone on the underlying fan $\Sigma$. For every cone $\sigma \in \Sigma$ we denote the associated closed toric subbundle by $\toricbun_{\sigma}$ and the associated torus bundle by $\Tbun_\sigma$. We introduce the filtrations
\[
    \B_{\sigma,\lambda} = \ov{\{ b \in \B(\Kbar)\mid \exists x \in \pi^{-1}(b)\cap{\toricbun_\sigma}, h_{\ov{\sL}}(x) \leq \lambda\}}.
\]
Then we obtain the height filtration on $\toricbun$ as
\[
    \toricbun_\lambda = \bigcup_{\sigma \in \Sigma} \pi^{-1}(\B_{\sigma, \lambda}) \cap \toricbun_\sigma.
\]

The last step remaining is to understand the filtrations $\B_{\sigma,\lambda}$. We have additional tools available to study fibres of toric bundles since they are toric varieties. For a point $b\in\B(\Kbar)$, we may consider the fibre $\toricbun_b$. Let $s:\toricvar\cong \toricbun_b$ be a trivialization. Here $\toricvar$ is a toric variety with torus $\T\subset \toricvar$.

Let $\ov{D}$ be a toric metrized Cartier divisor on $\toricvar$. Then, $\mrho_{\toricbun_b}(\ov{D})$ is the translate of $\ov{D}$ by $|s|^{-1} \in \T^{trop}$, when we trivialize the torsor by $s$. In particular, if $\ov{D}$ is semipositive the local roof function $\theta(\arho(\ov{D}))$ is $\theta(\ov{D})- \log |s|$ by \cite[Proposition 2.3.3]{toricheights}, where $\log|s|$ is the linear function $M\to \bbR$ defined by the norm of $s$.

If $\ov{D}$ is a semipositive adelic toric divisor its global roof function is $\theta(D)(m) + h_{\aTbun(m)}(b)$. This allows us to apply the results of \cite{sucmin} to fibers of $\pi$. We introduce the function $h_\Delta$ on $\B$ given by $h_\Delta(b)=\max_{m \in \Delta}\theta(b,m)$. The study of the toric filtration boils down to the study of $h_\Delta$ and the filtration by $\ov{\{ b \in \B(\Kbar)\mid  h_{\Delta}(b) \leq \lambda\}}$.

\section{\label{sec:okounkov}Okounkov bodies of toric bundles and applications}

Let $\B$ be a variety over $K$ of dimension $g$. Let $\toricbun$ be a toric bundle over $\B$ with underlying torus bundle $\Tbun$ and fan $\Sigma$. Let $\Delta$ be a polytope whose normal fan coarsens $\Sigma$ and $\rho(\Delta)$ the corresponding line bundle on $\toricbun$. Let $L$ be a line bundle on $\B$. We next compute the Okounkov body of $\rho(\Delta) + \pi^*L$.

In order to obtain a more convenient flag we refine $\Sigma$ such that it defines a smooth projective variety, see \cite[Chapter 11]{Cox_Little_Schneck_Toric}. Since the fibres $\toricvar_\Sigma$ are smooth we can take prime toric divisors $D_1,\dots,D_t$ such that
\[
\toricvar_\bullet:\toricvar_\Sigma\supset D_1 \supset D_1 \cap D_2 \supset\dots \supset D_1\cap \dots \cap D_t =\{p\}.
\]
defines a flag. We obtain a partial flag on $\toricbun$:
\[
\rho(\toricvar_\bullet):\toricbun \supset \rho(D_1) \supset \rho(D_1) \cap \rho(D_2) \supset\dots \supset \rho(D_1)\cap \dots \cap \rho(D_t) \cong \B.
\]
Given a flag $\B_\bullet$ on $\B$, we extend $\rho(\toricvar_\bullet)$ to a flag $\toricbun_\bullet$ on $\toricbun_\Sigma$. We now compute the Okounkov body of $\rho(\Delta) + \pi^*L$ with respect to $\toricbun_\bullet$.

By translating $\Delta$, we can always ensure that the conditions in Proposition \ref{prop:toric_okounkov} are satisfied. In the toric bundle setting this requires to change the line bundle $L$.

\begin{theorem}\label{thm:geometric_okounkov_body}
    The Okounkov body of $\toricbun$ with respect to $\rho(\Delta) + \pi^*L$ is given by the closure of
    \[
        \{(m,x)\mid m \in \Delta, x \in \Delta_{B_\bullet}(L+\Tbun(m))\}.
    \]
    For $m \in \Delta$ such that $L+\Tbun(m)$ is big the fibre is given by $\Delta_{B_\bullet}(L+\Tbun(m))$.
\end{theorem}
\begin{proof}
    Just as for complete flags one can associate a valuation $\nu_{\rho(\toricvar_\bullet)}$ to the partial flag $\rho(\toricvar_\bullet)$ as the first $t$ entries of the valuation associated to $\toricbun_\bullet$. By definition, the Okounkov body of $\rho(\Delta) + \pi^*L$ is the closure of the family of Okounkov bodies over $m\in \Delta\cap M_\bbQ$ of the linear series
\[
\im\left(\{s \in H^0(\toricbun, n(\pi^*L+\rho(\Delta))) \mid  \nu_{\rho(\toricvar_\bullet)}(s) \geq nm\} \overset{\alpha}{\to} H^0(X,n(L+\Tbun(m)))\right).
\]
We can understand this by applying Lemma \ref{prop:geometric_global_sections}. For $m\notin \Delta$ the image will be $0$ since then the domain of the map is $0$. On the other hand if $m \in \Delta \cap M_\bbQ$, the $\chi^m$-equivariant sections are in bijection with $H^0(\B, L + \Tbun(m))$. All non-zero $\chi^m$-equivariant sections $s$ satisfy $\nu_{\rho(\toricvar_\bullet)}(s) = m$. Therefore, the map $\alpha$ is surjective. Hence, the map of Okounkov bodies $\Delta_{\toricbun_\bullet}(\pi^*L+\rho(\Delta)))$ has fibres containing $\Delta_{B_\bullet}(L+\Tbun(m))$. For fixed $m \in \Delta$ such that $L+\Tbun(m)$ is big, the fibre of the closure of $\{(m,x)\mid m \in \Delta, x \in \Delta_{B_\bullet}(L+\Tbun(m))\}$ doesn't contain points outside $\Delta_{B_\bullet}(L+\Tbun(m))$ by the convexity of Okounkov bodies as a function of the line bundle, see \cite[Section 4.2]{lazmus}.
\end{proof}

\subsection{Arithmetic version}

Suppose that in addition to the setup we have an adelic structure $\aTbun$ on $\Tbun$. Denote by $\ac:M\to \aPic(\B)$ the homomorphism describing the isomorphism class of $\aTbun$. Suppose further that we have a semipositive toric adelic metric $\ov{D}$ with roof function $\theta$ on the divisor associated to $\Delta$ and that $\ov{L}$ is adelically metrized.

\begin{theorem}\label{thm:arithmetic_okounkov_body}
    The Boucksom-Chen transform $G_{\pi^*\ov{L}+\arho(\ov{D}),\toricbun_\bullet}(m,x)$ on 
    \[
        \{(m,x)\mid m \in \Delta, x \in \Delta_{B_\bullet}(L+\Tbun(m))\}
    \]
    is given by $\theta(m) + G_{\ov{L}+\aTbun(m),B_\bullet}(x)$ when $L+\Tbun(m)$ is big.
\end{theorem}

\begin{proof}
    The vector space
\[
\im\left(\{s \in H^0(\toricbun, n(\pi^*\ov{L}+\arho(\ov{D}))) \mid  \nu_{\rho(\toricvar_\bullet)}(s) \geq nm\} \overset{\alpha}{\to} H^0(\B,n(L+\Tbun(m)))\right).
\]
can be endowed with quotient norms at each, where the left hand side endowed with the supremum norm. The arithmetic Okounkov body of $\ov{L}+\arho(\ov{D})$ is the closure of the arithmetic Okounkov bodies for all the filtered graded linear series. By Lemma \ref{prop:metrized_global_sections}, the quotient norm for $\alpha$ agrees with the supremum norm on $H^0(X,n(\ov{L}+\aTbun(m)))$ twisted by $\theta_v(m)$.

By the convexity of arithmetic Okounkov bodies in families the theorem follows.
\end{proof}

\begin{remark}
    Due to the work of Sombra and Ballaÿ one will be able to deduce equidistribution of small points on some line bundles of the form $\arho(\ov{\Delta}) + \pi^*\ov{L}$ as they can phrase their sufficient condition on equidistribution in terms of the Boucksom-Chen transform, see \cite[Section 1.4]{ballaÿ2024approximationadelicdivisorsequidistribution}.
\end{remark}

\subsection{Application to successive minima}
In this section, we prove Theorem \ref{thm:minima} by studying the Boucksom-Chen transform.

\minima*

\begin{proof}
    If $m$ is such that $\ov{L} + \ac(m)$ is geometrically big, this holds in a neighborhood of $m$. It follows by Theorem \ref{thm:geometric_okounkov_body} that $\arho(\ov{\Delta}) + \pi^*\ov{L}$ is geometrically big. We can therefore apply Theorem \ref{thm:ballay} to compute the essential minimum in terms of the Boucksom-Chen transform. Recall that the Okounkov body of $\pi^*\ov{L}+\arho(\ov{D})$ maps to $\Delta$. Over $m \in \Delta$, the fibre can be identified with $\Delta_{B_\bullet}(L+\Tbun(m))$ and the restriction of the Boucksom-Chen transform $G_{\pi^*\ov{L}+\arho(\ov{D}),\toricbun_\bullet}$ is given by $\theta(m) + G_{\ov{L}+\aTbun(m),B_\bullet}(x)$. Hence, over each such fibre the maximum of the Boucksom-Chen function is $\zeta_{\ess}(\ov{L} + \ac(m))+\theta(m)$ by Theorem \ref{thm:ballay}.
    
    Let $\Delta^{\text{big}}$ denote the locus of $m \in \Delta$ such that $L+c(m)$ is big. By concavity and upper semicontinuity of the Boucksom-Chen transform and concavity of the essential minimum it follows that 
    \begin{align*}
        \zeta_{\ess}(\arho(\ov{\Delta}) + \pi^*\ov{L})&=\max_{(m,x) \in \Delta_{\toricbun_\bullet}(\pi^*L+\rho(D))} G_{\pi^*\ov{L}+\arho(\ov{D}),\toricbun_\bullet}(m,x)\\ &=\sup_{m\in\Delta^{\text{big}}}\zeta_{\ess}(\ov{L} + \ac(m))+\theta(m)\\ &= \sup_{m\in\Delta}\zeta_{\ess}(\ov{L} + \ac(m))+\theta(m).
    \end{align*}
    We now assume that in addition $\ov{L} + \ac(m)$ is semipositive. In order for $L + c(m)$ to admit semipositive metrics it has to be nef. Hence, the Hilbert-Samuel theorem \cite[Proposition 1.31]{debarre_higher_dimensional_alggeo} holds. Adding an arbitrarily small multiple of a big nef line bundle yields a big nef divisor. Hence, $L + c(m)$ is pseudoeffective for all $m\in\Delta$ and big for all $m$ in the interior $\Delta^\circ$.

    By the upper semi-continuity of the Boucksom-Chen transform and applying Theorem \ref{thm:ballay} to each fibre, we obtain
    \begin{align*}
        \zeta_{\abs}(\arho(\ov{\Delta}) + \pi^*\ov{L})&=\inf_{(m,x) \in \Delta_{\toricbun_\bullet}(\pi^*L+\rho(D))} G_{\pi^*\ov{L}+\arho(\ov{D}),\toricbun_\bullet}(m,x)\\ &=\inf_{m\in\Delta^\circ}\zeta_{\abs}(\ov{L} + \ac(m))+\theta(m).
    \end{align*}
\end{proof}

\section{\label{hkm}Arithmetic bundle BKK}

The purpose of this section is to prove the arithmetic bundle BKK theorem stated in the introduction. Its proof will follow the outline of the proof of \cite[Theorem 4.1]{toricbundles}. Let us swiftly recall the statement of the theorem.

Let $\atoricbun$ be an adelic integrable projective toric bundle of relative dimension $t$ with smooth generic fibre over a smooth projective base variety $\B$ of dimension $g$. Let $\sB$ be a flat, projective model of $\B$ over $\sS$. Let $A^*(\blank)$ denote the operational arithmetic Chow cohomology introduced in Section \ref{subsec:chow_homology}. Let $\gamma \in A^{g+1-i}(\sB)$ and denote by $[\infty]\in \ArChow^1(\sB)_\bbR$ the class of a trivial Cartier divisor endowed with constant Green's functions at all places such that $h_{[\infty]}(x) = 1$ for all $x \in \B(\Kbar)$.

\arithmeticbundlebkk*

\begin{remark}
    The above intersection numbers are shown to be well-defined in Section \ref{subsec:chow_homology}. In fact, by approximation the formula also holds for expressions $\gamma$ of the form $\ac_1(\sL_1)\dots\ac_1(\sL_r)\gamma'$ for integrable adelic line bundles $\sL_1,\dots,\sL_r$ on $\B$ and $\gamma' \in \A^{g+1-i-r}(\sB)$. When $\gamma$ is the product of first Chern classes of integrable line bundles the result follows from Section \ref{sec:okounkov}.

    One can prove the statement in further situations. When the torus bundle has a model over $\sB$, the result holds for $\gamma \in \ArChow_i(\sB)$. For a result of similar homological flavour in the general case one may use $b$-cycles satisfying positivity conditions as introduced in Section \ref{subsec:chow_homology}.
\end{remark}

We introduce shorthand notations for use in the course of the proof. Let
\[
    \aI_\gamma:\aP^+ \to \bbR,\ \ov{\Delta} \mapsto \int_\Delta (\widehat{c}(m) + \theta(m)[\infty])^i \cdot \gamma dm
\]
and
\[
    \aF_\gamma:\aP \to \bbR,\ \ov{\Delta} \mapsto \adeg(\arho(\ov{\Delta})^{t+i}\cdot \pi^*\gamma).
\]
The function $\aF_\gamma$ is well-defined as the intersection number does not change under birational modification. We eventually extend $\aI_\gamma$ to the space of all virtual polytopes. Using the introduced shorthand, the theorem is stated below.

\begin{theorem}
  \label{bkk}The polynomials $\aI_\gamma$ and $\aF_\gamma$ satisfy
\eqn{(t+i)!\cdot \aI_\gamma(\ov{\Delta}) = i! \cdot \aF_\gamma(\ov{\Delta}).}
In particular, the polarizations of $\aI_\gamma$ and $\aF_\gamma$ are proportional multilinear forms, i.e.\ for any $\ov{\Delta}_1, . . . , \ov{\Delta}_{t+i} \in \aP_\Sigma$
\eqn{(t+i)!\cdot \aI_\gamma(\ov{\Delta}_1, . . . , \ov{\Delta}_{t+i}) = i! \cdot \aF_\gamma(\ov{\Delta}_1, . . . , \ov{\Delta}_{t+i}).}  
\end{theorem}

This closely resembles the BKK-type theorem in \cite{toricbundles}. Let us recall their statement for context. Let $\toricbun \to \B$ be a toric bundle over a smooth compact oriented $\bbR$-manifold $\B$ with smooth underlying fan $\Sigma$. Analogous to our situation there are maps $c:M \to H^2(\B,\bbZ)$ and $\rho:\calP_\Sigma \to H^2(\B,\bbZ)$. Furthermore, the toric bundle $\toricbun$ has an induced orientation. Hence, the top cohomology groups will be identified with $\bbR$.

\begin{theorem}[Theorem 4.1 \cite{toricbundles}]
    Let $\gamma\in H^{k-2i}(\B,\bbR)$. Then, the cup product satisfies
    \[
        i!\rho(\Delta)^{t+i}\pi^* \gamma = (t+i)! \int_\Delta c(m)^i\cdot \gamma dm.
    \]
\end{theorem}

\subsection{Overview of proof}

Following the outline in \cite{toricbundles} we prove Theorem \ref{bkk} by first showing that the two functions are polynomials and comparing their partial derivatives. The novel notion of adelic polytopes allows us to transfer many ideas from the classical setting. We obtain that $\aI_\gamma$ is a homogeneous polynomial function in Corollary \ref{cor:aIispolynomial}. It is clear, that the same holds for $\aF_\gamma$. For $i>0$, it suffices to show
\[
i!\partial^{k_1}_{1}\dots\partial^{k_r}_{r} \aF_\gamma(\ov{\Delta})=(t+i)!\partial^{k_1}_{1}\dots\partial^{k_r}_{r} \aI_\gamma(\ov{\Delta})
\]
for partial derivatives along the rays of $\widetilde{\Sigma}$ of total degree $k_1+\dots+k_r = t+1$. In the case of multiplicity $\mu=k_1+\dots+k_r-r$ of the partial derivative being $0$, the comparison is done by direct calculation. We perform induction on $i+\mu$.

For $i=0$, we observe that
\eqn{
t!\cdot \aI_\gamma(\ov{\Delta}) = t!\cdot \vol(\Delta) \deg(\gamma)= \deg(\Delta) \adeg(\gamma)= \adeg(\arho(\ov{\Delta})^{t}\cdot \pi^*\gamma) = \aF_\gamma(\ov{\Delta}).
}
This is just the classical BKK theorem except for the second to last equality which follows from a projection formula that can be deduced adhoc.

The cycle $\gamma$, or more precisely $\gamma\cap [\sB]$, can be represented as a sum of closed points and a measure on $\sB$. By linearity, assume first that $\gamma$ is represented by a closed point. The cycle $\pi^*\gamma$ is given as the fibre over $\gamma$. The restriction of $\rho(\Delta)$ to this fibre has degree $\deg(\Delta)$. Hence, the claim holds in this case. Now suppose $\gamma$ has only an archimedean part $\omega$. Then $\rho(\Delta)^{t}\cdot \pi^*\gamma = \int_{\toricbun(\bbC)} c_1(\arho(\Delta))^t \pi^* \omega = \int_{\B(\bbC)}\left(\int_{\pi^{-1}(b)} c_1(\arho(\Delta))^t|_{\pi^{-1}(b)} \right)\omega(b) = \deg(\Delta) \deg{\gamma}$.

\subsection{Arithmetic convex chains}

The goal of this section is to prove that $\aI_\gamma(\ov{\Delta})$ is a polynomial on the space of arithmetic virtual polytopes. This requires an extension of the ideas in \cite{pukhlikov_khovanskii_virtual_polyhedra} to the arithmetic setting.

\begin{definition}
    Let $V$ be a finite dimensional real vector space. Then, a convex chain on $V$ is a function $\alpha:V \to \bbZ$ of the form $\alpha = \sum_{i=1}^k n_i \one_{A_i}$ for polytopes $A_i \in \calP^+$ and $n_i \in \bbZ$. This forms an algebra $Z(V)$ under the usual addition and multiplication given by the convolution product.
\end{definition}

\begin{definition}
    A \emph{finitely additive measure} on $\calP^+$ is a map $\phi:\calP^+ \to M$ to an abelian group $M$ satisfying the following property:

    If $A_1,\dots,A_N\in \calP^+$ are such that $\cup_{i=1}^N A_i \in \calP^+$, then the following inclusion-exclusion relation holds:
    \[
        \phi(\bigcup_{i=1}^N A_i) = \sum_i \phi(A_i) - \sum_{i<j} \phi(A_i\cap A_j)+\dots
    \]
    The empty set satisfies $F(\emptyset) = 0$.
\end{definition}

\begin{definition}
    \begin{enumerate}
        \item A map $p:N\to M$ of abelian groups is called a polynomial of degree $\leq k$ if one of the following two conditions holds:
        \begin{enumerate}
            \item $k=0$ and $p$ is constant, i.e.\ $p(N)=m \in M$
            \item $k\geq 1$ and for any $a \in N$, the map $p_a: N \to M$, $p_a:x\mapsto p(x+a)-p(x)$, is a polynomial of degree $\leq k-1$.
        \end{enumerate}
        \item A measure $\phi: \calP^+ \to M$ is polynomial of degree $\leq k$ if for each $A \in \calP^+$ the function $\phi(A + v):V \to M$ is polynomial of degree $\leq k$.
    \end{enumerate}
\end{definition}

\begin{remark}
    The notion of convex chains allows for a reinterpretation of the notions of measure. Namely, a finitely additive measure is an arbitrary homomorphism of additive groups $\phi:Z(V) \to M$. If the measure is polynomial of degree $\leq k$, this extends to translation of functions in $Z(V)$. The remark justifies calling $\calP^+ \to Z(V)$ the universal measure. This is discussed below \cite[Definition 2.8]{pukhlikov_khovanskii_virtual_polyhedra}.
\end{remark}

Let $\tau_v:Z(V) \to Z(V)$ be the translation by a vector defined by $\tau_v \alpha (x) = \alpha(x-v)$ for $\alpha\in Z(V)$. Let $J_k \subset Z(V)$ denote the subgroup generated by chains of the form
\[
    (\tau_{v_1} - 1) \circ \dots \circ (\tau_{v_k} - 1)(\alpha)
\]
for all $v_1, \dots,v_k \in V$. The subgroup $J_k \subset Z(V)$ is an ideal. The map $\calP^+ \to Z(V)/J_{k+1}$ is the universal polynomial measure of degree $\leq k$, i.e.\ all polynomial measures of degree $\leq k$ factor uniquely through a homomorphism $Z(V)/J_{k+1} \to M$.

The degree of a convex chain $\alpha = \sum_{i=1}^k n_i \one_{A_i}$ for $A_i \in \calP^+$ is defined to be $\sum_{i=1}^k n_i$. This is well-defined by \cite[Proposition/Definition 2.1]{pukhlikov_khovanskii_virtual_polyhedra}. Let $\scrL \subset Z(V)$ denote the ideal of degree $0$ chains.

\begin{theorem}[Theorem 2.3 \cite{pukhlikov_khovanskii_virtual_polyhedra}]\label{thm:pukh_khov}
    For $k \geq 1$,
    \[
        \scrL^{\dim V + k} \subset J_k.
    \]
\end{theorem}
This theorem that any polynomial measure on $V$ of degree $\leq k$ restricted to the group of virtual polytopes $\calP$ is a polynomial of degree $\leq \dim V + k$ by \cite[Corollary 2.5]{pukhlikov_khovanskii_virtual_polyhedra}.

Let us now introduce the arithmetic analogues of $\calP$. Due to the view towards toric varieties, we denote $\dim V = t$. We freely use the notation of Definition \ref{defi:I_metrized_polytope} and

\begin{definition}
    An $I$-metrized convex chain on $V$ is a function $\alpha:V\oplus \bigoplus_{i \in I} \bbR \to \bbZ$ of the form $\alpha = \sum_{i=1}^k n_i \one_{A_i}$ for $I$-metrized polytopes $A_i \in \calP^{I,+}$ and $n_i \in \bbZ$. This forms an algebra $Z^I(V)$ under the usual addition and multiplication given by the convolution product. For $I=\emptyset$, we recover the algebra of convex chains $Z(V)$.
\end{definition}

Let $*$ denote a one element set. Then, the addition $\bigoplus_{i \in I} \bbR \to \bbR$ via the pushforward from \cite[Proposition/Definition 2.2]{pukhlikov_khovanskii_virtual_polyhedra} induces a ring homomorphism $Z^I(V) \to Z^*(V)$. We note that the algebra of metrized convex chains $Z^*(V)$ can be is a subalgebra of $Z(V \oplus \bbR)$. The $I$-metrized convex chains are a subalgebra of $\bigcup_{J\subseteq I,\text{finite}} Z(V \oplus \bigoplus_{j \in J} \bbR)$.

A polynomial map from an $\bbR$-vector $W$ space to $\bbR$ will be called a polynomial function if it is continuous on every finite dimensional subvector space. A polynomial function $f:W\to\bbR$ is said to be homogeneous of degree $k$ if for $\lambda \in \bbR$ and $w\in W$ the equality $f(\lambda w)= \lambda^kf(w)$ holds.

\begin{definition}
    Let $f$ be a polynomial function of degree $\leq k$ on $V$. Then, we denote by $I_f$ the map $\calP \to \bbR$ extended from $\Delta \mapsto \int_\Delta f$. It is a degree $\leq k$ measure on the space of polytopes.
\end{definition}

The statement of \cite[Theorem 5.5]{toricbundles} is an easy corollary of Theorem \ref{thm:pukh_khov} and summarizes the discussion in a convenient way.

\begin{theorem}[Theorem 5.5 \cite{toricbundles}]
    If $f:V\to \bbR$ is a homogeneous polynomial function of degree $k$, then the function $I_f: \calP^+ \to \bbR; (\Delta) \mapsto I_f(\Delta) = \int_\Delta f(x) d\mu$ admits a unique extension to a homogeneous polynomial function of degree $t+k$ on $\calP$.
\end{theorem}

We will apply this to prove an arithmetic variant. In the proof of Theorem \ref{bkk} it will be applied for the function $M_\bbR \times \bbR\to \bbR$ given by $(m,x) \mapsto (\ac(m) + x[\infty])^i\gamma$.

\begin{theorem}\label{thm:arithmetic_polynomial_extension}
    If $f:V \times \bbR \to \bbR$ is a homogeneous polynomial function of total degree $k$, then the function 
    \[
        \aI_f: \aP^+ \to \bbR,\ \ov{\Delta} \mapsto \int_\Delta f(m,\theta(m)) dm
    \]
    admits a unique extension to a homogeneous polynomial function of degree $t+k$ on $\aP$. Here $\theta$ denotes the global roof function of $\ov{\Delta}$.
\end{theorem}

\begin{proof}
    We reduce the statement to \cite[Theorem 5.5]{toricbundles}. For this note that the partial derivative $f'$ with respect to the last variable is a polynomial of degree $k-1$. Assume that $\theta_v(x)\geq 0$ for all $x\in \Delta$ and all places $v\in M_\K$. Then, we have that $\aI_f(\ov{\Delta}) = I_f(\Delta) + I_{f'}(\widehat{\Delta})$ by the fundamental theorem of calculus. The first term is known to be a polynomial of degree $t+k$ by \cite[Theorem 5.5]{toricbundles}.

    The second term is a degree $\leq k-1$ measure on $Z^*(V)$. Hence, it gives a degree $\leq k+t$ polynomial on $Z^*(V)$. Since the map $Z^{M_K}(V) \to Z^*(V)$ is a ring homomorphism, it follows that we obtain a degree $t+k$ polynomial on virtual adelic polytopes.
\end{proof}

\begin{corollary}\label{cor:aIispolynomial}
    The function
    \[
        \aI_\gamma: \aP^+ \to \bbR,\ \ov{\Delta} \mapsto \int_\Delta (\ac(m) + \theta(m)[\infty])^i\gamma dm
    \]
    admits a unique extension to a homogeneous polynomial function of degree $t+i$ on $\aP$.
\end{corollary}

\begin{proof}
    We apply Theorem \ref{thm:arithmetic_polynomial_extension} to the function $M_\bbR \times \bbR\to \bbR$ given by $(m,x) \mapsto (\ac(m) + x[\infty])^i\gamma$.
\end{proof}

\subsection{Differentiation of $\aI_\gamma$}
We compute the $(t+1)$-st partial derivatives of $\aI_\gamma$ of multiplicity $0$ for a preferred basis of the space of adelic polytopes. We compute more generally the derivatives of functions of the form $\aI_f$ introduced in Theorem \ref{thm:arithmetic_polynomial_extension} for smooth $f$ on $V\oplus\bbR$. Denote the restriction of $f$ to $V \times 0$ by $f$ as well and $f'$ the partial derivative along the $\bbR$-summand. We note by the proof of Theorem \ref{thm:arithmetic_polynomial_extension} that $\aI_f(\ov{\Delta}) = I_f(\Delta) + I_{f'}(\widehat{\Delta})$. We compute the differentials of each summand separately.

Let $\widetilde{\Sigma}$ be a simplicial adelic fan with recession fan $\Sigma$. The vector space of virtual adelic polytopes compatible with $\widetilde{\Sigma}$, $\aP_{\widetilde{\Sigma}}$, has a distinguished basis provided by the virtual polytopes corresponding to the rays of $\widetilde{\Sigma}$. We say that a set of rays spans a cone if there is a place at which they span a cone. Since $\widetilde{\Sigma}$ is simplicial, there is a unique virtual polytope $h$ whose support function is $1$ at the primitive generator of a chosen ray $\tau$ and $0$ on all other rays. We refer to $h$ as the polytope dual to $\tau$. We will be working with a finite set of rays $\tau_1,\dots,\tau_s$. We denote its primitive generators by $e_1,\dots,e_s$ and its dual polytopes by $h_1,\dots,h_s$. We denote the partial derivative in the direction of $h_i$ by $\partial_i$.

We recall first the classical case. Let $f:V \to \bbR$ be a continuous function and $\Sigma$ be a complete fan on $V$. Let $\Delta$ be a polytope in the interior of $\calP^+_\Sigma$.

\begin{lemma}[Lemma 6.1 \cite{toricbundles}]
    Let $\tau_1,\dots\tau_{t}$ be rays spanning a maximal cone dual to a vertex $A \in \Delta$. Then, we have
    \eqn{\partial_1\dots\partial_{t}(I_f)(\Delta) = f(A)\cdot |\det(e_1,\dots,e_t)|.}
\end{lemma}

Note that in the reference $f$ is assumed smooth. This is, however, not used in the proof. An adelic analogue of \cite[Lemma 6.1]{toricbundles} is given below.

\begin{proposition}
    \label{lemmaI}Let $\tau_1,\dots\tau_{t+1}$ be rays spanning a maximal cone $\sigma$ in $\widetilde{\Sigma}$ at a place $v$ and $\ov{\Delta}\in\aP^+_{\widetilde{\Sigma}}$ a $v$-interior polytope. Suppose that $\sigma$ is dual to the vertex $(A,\theta_v(A))\in \Delta_v$. Then, we have
    \eqn{\partial_1\dots\partial_{t+1}(I_f)(\widehat{\Delta}) = f(A, \theta(A))\cdot |\det(e_1,\dots,e_{t+1})|.}
\end{proposition}

\begin{proof}
    Extend the function $\theta - \theta_v$ to a continuous function $\sE$ on $V$. Define $\widetilde{f}(v,t) = f(v,t+\sE(v))$. Then one easily sees that $\partial_1\dots\partial_{t+1}(I_f)(\widehat{\Delta}) = \partial_1\dots\partial_{t+1}(I_{\widetilde{f}})(\Delta_v)$. We can then apply \cite[Lemma 6.1]{toricbundles} to find that the derivative is
    \[\widetilde{f}(A,\theta_v(A))\cdot |\det(e_1,\dots,e_{t+1})| = f(A, \theta(A))\cdot |\det(e_1,\dots,e_{t+1})|.\]
\end{proof}

If $\ov{\Delta}$ is a polytope in $\aP^+_{\widetilde{\Sigma}}$ and $\tau_{1},\dots,\tau_{r}$ do not span a cone in $\widetilde{\Sigma}$, then we have
    \eqn{\partial^{k_1}_{1}\dots\partial^{k_r}_{r}(I_f)(\ov{\Delta}) = 0}
for any tuple of $k_i \geq 1$.

\begin{corollary}
    \label{cor:aI}Let $\tau_1,\dots\tau_{t+1}$ be rays spanning a maximal cone $\sigma$ in $\widetilde{\Sigma}$ at a place $v$ and $\ov{\Delta}\in\aP^+_{\widetilde{\Sigma}}$ a $v$-interior polytope. Suppose that $\sigma$ is dual to the vertex $(A,\theta_v(A))\in \ov{\Delta}_v$. Then, we have
    \eqn{\partial_1\dots\partial_{t+1}(\aI_\gamma)(\ov{\Delta}) = i\ac(A)^{i-1}\cdot[\infty]\cdot\gamma \cdot |\det(e_{i_1},\dots,e_{i_r})|.}
\end{corollary}

\begin{proof}
    Recall the decomposition $\aI_f(\ov{\Delta}) = I_f(\Delta) + I_{f'}(\widehat{\Delta})$. We observe that
    \[
        \partial_1\dots\partial_{t+1} I_f(\Delta) = 0
    \]
    since at least one of the rays does not lie in $V \times 0$. Note that $\aI_\gamma = \aI_f$ for $f(m,t)=(\ac(m) + t [\infty])^i\gamma$.
    
    We compute $f'(m,t) = i(\ac(x) + \theta [\infty])^{i-1}[\infty]\gamma = i\ac(x)^{i-1}[\infty]\gamma$. The last equality follows since $[\infty]^2 = 0$ in the Chow ring.
\end{proof}

If $\ov{\Delta}$ is a polytope in $\aP^+_{\widetilde{\Sigma}}$ and $\tau_{1},\dots,\tau_{r}$ do not span a cone in $\widetilde{\Sigma}$, then we have
    \eqn{\partial^{k_1}_{1}\dots\partial^{k_r}_{r}(\aI_\gamma)(\ov{\Delta}) = 0}
for any tuple of $k_i \geq 1$.

\subsection{Differentiation of $\aF_\gamma$}
Let us consider first the squarefree case.

\begin{lemma}
    \label{lemmF}
    Let $\tau_1,\dots\tau_{t+1}$ be rays spanning a maximal cone $\sigma$ in $\widetilde{\Sigma}$ at a place $v$ and $\ov{\Delta}\in\aP^+_{\widetilde{\Sigma}}$ a $v$-interior polytope. Suppose that $\sigma$ is dual to the vertex $(A,\theta_v(A))\in \ov{\Delta}_v$. Then, 
    \eqn{\partial_1\dots\partial_{t+1}(\aF_\gamma)(\ov{\Delta}) = \frac{(t+i)!}{i!}\cdot i \ac(A)^{i-1}[\infty]\gamma\cdot |\det(e_{i_1},\dots,e_{i_r})|.}
    
    If $\ov{\Delta}$ is a polytope in $\aP^+_{\widetilde{\Sigma}}$ and $\tau_{1},\dots,\tau_{r}$ do not span a cone in $\widetilde{\Sigma}$, then we have
    \eqn{\partial^{k_1}_{1}\dots\partial^{k_r}_{r}(\aF_\gamma)(\ov{\Delta}) = 0}
    for any tuple of $k_i \geq 1$.
\end{lemma}

\begin{proof}
    For the vanishing result it is clearly sufficient to consider the case $k_i = 1$ for all $i$. Denote by $\aD_i$ the divisor associated to $h_i$ on $\toricbun_{\widetilde{\Sigma}}$.

    We expand the polynomial $F_\gamma$ at $\Delta$ in order to compute the derivative.
    \begin{align*}
        \aF_\gamma(\ov{\Delta} + \sum_i \lambda_i \arho(\aD_i)) &= (\arho(\ov{\Delta}) + \sum_i \lambda_i \arho(\aD_i))^{t+i}\pi^*(\gamma)\\
        &= \sum_{\alpha_0 + \dots + \alpha_s = t+i} \binom{t+i}{\alpha_0, \dots, \alpha_s}\arho(\ov{\Delta})^{\alpha_0}\arho(\aD_1)^{\alpha_1}\cdots\arho(\aD_r)^{\alpha_r} \pi^*(\gamma)\lambda^{\alpha_1}_1 \cdots \lambda^{\alpha_r}_r.
    \end{align*}
    The expression $\partial_1\dots\partial_r(\aF_\gamma)(\Delta)$ shows up as the coefficient of $\lambda_1 \cdots \lambda_r$. Since the intersection of the $\aD_i$ is empty it follows by Proposition \ref{empty_intersection} that the arithmetic intersection also vanishes.

    We proceed to the case, where $\tau_1,\dots\tau_{t+1}$ span a cone. We have
    \eqn{
     \partial_1\dots\partial_{t+1}F_\gamma(\Delta) = \frac{(t+i)!}{i!} \arho(\ov{\Delta})^i  \arho(\widehat{D}_1)\cdots\arho(\widehat{D}_{t+1}) \pi^*(\gamma).
    }
Let $\widetilde{\Delta}$ be the virtual adelic polytope $\ov{\Delta}-(A,\theta_v(A))$. Then, 
\eqn{h_{\widetilde{\Delta}}(e_1)= \dots = h_{\widetilde{\Delta}}(e_{t+1}) = 0}
since the vertex of $\widetilde{\Delta}$ corresponding to $A$ is sent to the origin. Therefore, $\widetilde{\Delta}$ is the linear combination of rays not belonging to $\sigma$. In particular,
\eqn{\arho(\widetilde{\Delta})\arho(\widehat{D}_1)\cdots\arho(\widehat{D}_{t+1}) = 0.}
We compute
\begin{align*}
    &\arho(\ov{\Delta})^i \arho(\widehat{D}_1)\cdots\arho(\widehat{D}_{t+1}) \pi^*(\gamma)\\=& \arho(\widetilde{\Delta} + (A,\theta_v(A))^i\arho(\widehat{D}_1)\cdots\arho(\widehat{D}_{t+1}) \pi^*(\gamma)\\=& \arho((A,\theta_v(A)))^i\arho(\widehat{D}_1)\cdots\arho(\widehat{D}_{t+1}) \pi^*(\gamma).
\end{align*}
We now apply that $\arho(A,\theta_v(A)) = \pi^* \widehat{c}(A,\theta_v(A))$ and an explicit projection formula. We restrict to the case of model metrics by approximation. The intersection of $\widehat{D}_1,\dots,\widehat{D}_{t+1}$ on $\toricvar_{\Sigma}$ is given as the $|\det(e_{i_1},\dots,e_{i_t})|$-multiple of a closed point at the place $v$. In particular, $\arho(\widehat{D}_1)\dots\arho(\widehat{D}_{t+1})$ is the multiple of a horizontal cycle mapping to the special fibre at $v$ of $\sB$. Hence,
\[
\arho(\widehat{D}_1)\cdots\arho(\widehat{D}_{t+1})\pi^*(\widehat{c}(A,\theta_v(A))^i \gamma)
\]
is represented by the multiple of a cycle that maps one to one to the intersection of $\widehat{c}(A,\theta_v(A))^i \gamma$ with the special fibre at a place. The archimedean case follows by the projection formula for differential forms.
\end{proof}

\begin{corollary}
    For any $i \leq g+1$ and $\gamma \in \widehat{CH}^{g+1-i}(A)$ and any squarefree differential monomial $\partial_I$ of order $t+1$, we have
\eqn{(t+i)!\cdot\partial_I \aI_\gamma(\ov{\Delta}) = i!\cdot\partial_I \aF_\gamma(\ov{\Delta}).}
\end{corollary}

We now need to consider partial derivatives which are not squarefree. We treat them by induction on the multiplicity. The multiplicity of a multiset is the difference of the cardinality of the multiset and the cardinality of the underlying set.

Let $I$ be a multiset of rays in $\widetilde{\Sigma}$. As induction hypothesis we assume that $(t+i)!\cdot\partial_I \aI_\gamma(\Delta) = i!\cdot\partial_I \aF_\gamma(\Delta)$ for all differential monomials $\partial_I$ of multiplicity $\mu-1 \geq 0$. We need to show the same equality for multiplicity $\mu$. Let $\tau_1,\dots,\tau_r$ be rays in $\widetilde{\Sigma}$ forming a cone. By relabeling, we may assume $\partial_I = \partial^{k_1}_1 \dots \partial^{k_r}_r$ and $k_1 > 1$. We may restrict to the case where $\tau_1,\dots, \tau_r$ form a cone in $\widetilde{\Sigma}$.

We express $\partial_1$ in terms of a Lie derivative $L_v$ for $v \in M_{\bbR} \times \bbR$ and other partial derivatives. As $e_1,\dots, e_r$ generate a cone in the simplicial fan $\widetilde{\Sigma}$ they can be completed to a basis $e_1,\dots, e_{t+1}$ of $N_\bbR \oplus \bbR$. The first vector of the dual basis satisfies $\langle v,e_1\rangle = 1$ and $\langle v,e_j\rangle = 0$ for $j=2, \dots,r$. The vector $v$ is of the form $\sum^s_{i=1} \langle v,e_i\rangle h_i$ for further prime virtual polytopes $h_{r+1},\dots,h_s$. We conclude that $L_v = \sum^s_{i=1} \langle v,e_i\rangle \partial_i$ and thus $\partial_1 = L_v - \sum_{j>r} \langle v,e_i\rangle \partial_j$. We get:
\begin{align*}
\partial_I = \partial^{k_1}_1 \dots \partial^{k_r}_r &= \left( L_v - \sum_{j>r} \langle v,e_i\rangle \partial_j \right)\partial^{k_1 -1}_1 \dots \partial^{k_r}_r\\
&=L_v \partial^{k_1 -1}_1 \dots \partial^{k_r}_r - \sum_{j>r} \langle v,e_i\rangle \partial^{k_1 -1}_1 \dots \partial^{k_r}_r\partial_j.
\end{align*}
We may apply the induction hypothesis to the second term.

We use induction on $i$ to show the following statement analogous to \cite[Lemma 6.5]{toricbundles}.

\begin{lemma}
    In the above situation, we have $(t+i)!\cdot L_v \aI_\gamma(\ov{\Delta}) = i!\cdot L_v \aF_\gamma(\ov{\Delta})$.
\end{lemma}

\begin{proof}
    We prove the statement by two direct computations. Write $v = (x,t)$.
    \begin{align*}
        L_v I_\gamma(\ov{\Delta}) &= \partial_s|_{s=0} \int_{\Delta +sx} \widehat{c}(m + \theta(x) +sm)^i\gamma dm\\ &= \partial_s|_{s=0} \int_{\Delta} \widehat{c}(m +sv)^i\gamma dm\\
        &=\partial_s|_{s=0} \int_{\Delta} \sum^i_{j=0} {i\choose j }s^j\widehat{c}(v)^j\widehat{c}(m)^{i-j} \gamma dm\\
        &=i\int_{\Delta} \widehat{c}(v)\widehat{c}(m)^{i-1}\gamma dm = i \aI_{\widehat{c}(v)\gamma}(\ov{\Delta}).
    \end{align*}
    On the other hand,
    \begin{align*}
        L_v \aF_\gamma(\ov{\Delta}) &= \partial_s|_{s=0} \arho(\ov{\Delta} +sv)^{t+i} \gamma = \partial_s|_{s=0} \sum^{t+i}_{j=0} {{t+i}\choose j} s^j\pi^*\widehat{c}(v)^j\arho(\ov{\Delta})^{t+i-j} \pi^*\gamma\\
        &=(t+i)\arho(\ov{\Delta})^{t+i-1} \pi^*(\widehat{c}(v)\gamma )\\
        &=(t+i)F_{\widehat{c}(v)\gamma}(\ov{\Delta}).
    \end{align*}
    By induction hypothesis,
    \[
    (t+i-1)! \aI_{\widehat{c}(v)\gamma}(\ov{\Delta}) = (i-1)!\aF_{\widehat{c}(v)\gamma}(\ov{\Delta}).
    \]
    Hence,
    \[
    (t+i)! L_v \aI_\gamma(\ov{\Delta}) = i (t+i)! \aI_{\widehat{c}(v)\gamma}(\ov{\Delta}) = (t+i) i! \aF_{\widehat{c}(v)\gamma}(\ov{\Delta}) = i! L_v \aF_\gamma(\ov{\Delta}).
    \]
\end{proof}

This finishes the proof.

\section{\label{sec:examples}Semiabelian varieties}

In this section we study the semiabelian varieties and their compactifications by applying the methods developed earlier in the article. This allows us to recover Chambert-Loir's computation of the height and the absolute minimum in \cite{cl2}.

The first result is the calculation of the Okounkov body and Boucksom-Chen transform for semiabelian varieties. Let us recall the setup in the introduction. 

Let $G$ be a semiabelian variety over a global field $\K$ with abelian quotient $A$ and split torus part $\T$. Suppose the isomorphism class as a torus bundle is given by a map $c:M \to A^\vee(\Kbar)$. Let $\sM$ denote an ample symmetric line bundle on $A$. Endowing it with canonical metrics gives it the structure of an adelically metrized line bundle $\ov{\sM}$. Let $\widehat{h}$ denote the N\'eron-Tate height with respect to $\sM$ on $A(\Kbar)\otimes \bbR$ as well as the induced map on $A^\vee(\Kbar)\otimes \bbR$ along the polarization $A(\Kbar)\otimes \bbR \to A^\vee(\Kbar)\otimes \bbR$. Let $\ov{G}$ be the compactification of $G$ with respect to a fan $\Sigma$ in $N_\bbR$. Let $D$ be an ample toric Cartier divisor on $\toricvar_\Sigma$ with Newton polytope $\Delta \subset M_\bbR$. After refining $\Sigma$, we assume it defines a smooth projective variety. Let $\ov{G}_\bullet$ be a flag on $\ov{G}$ of the type considered in Section \ref{sec:okounkov} consisting of a toric part $\toricvar_\bullet$ and an abelian part $A_\bullet$.

\begin{lemma}\label{semiabelian_okounkov_body}
    The Okounkov body of $\ov{G}$ with respect to $\pi^*\sM+\rho(D)$ decomposes as a product $\Delta\times \Delta_{A_\bullet}(\sM)$. The Boucksom-Chen transform satisfies 
    \[
    G_{\ov{G}_\bullet,\pi^*\ov{\sM}+\rho(D)^{\can}}(m,x) = - \widehat{h}(c(m)).
    \]
\end{lemma}

\begin{proof}
    By Theorem \ref{thm:geometric_okounkov_body}, the fibers of the Okounkov body over $m \in \Delta$ are given by the Okounkov body $\Delta_{A_\bullet}(\sM+\sQ)$ for some numerically trivial line bundle $\sQ$ on $A$. By \cite[Proposition 4.1]{lazmus}, there is an equality of Okounkov bodies $\Delta_{A_\bullet}(\sM+\sQ) = \Delta_{A_\bullet}(\sM)$.

    In order to compute the Boucksom-Chen transform, we need to compute the Boucksom-Chen transform $G_{A_\bullet,\ov{\sM}+\ov{\sQ}}(x)$ for a numerically trivial line bundle $\sQ$ with its canonical metric. We note that $\ov{\sM}+\ov{\sQ}$ is semipositive and that $\zeta_{\ess}(\ov{\sM}+\ov{\sQ})=\zeta_{\abs}(\ov{\sM}+\ov{\sQ})$, since heights with respect to $\ov{\sM}+\ov{\sQ}$ are invariant under translation by the Zariski dense set of torsion points of $A$. By the results in Section \ref{prelim_okounkov}, the Boucksom-Chen transform is constant of value
    \[
        \frac{\adeg((\ov{\sM}+\ov{\sQ})^{g+1})}{(g+1)\deg((\sM+\sQ)^{g})}.
    \]
    Suppose that $\sQ$ is the image of $q\in A(\Kbar)$ under the polarization morphism $A \to A^\vee$. The intersection number $\adeg((\ov{\sM}+\ov{\sQ})^{g+1})$ is calculated in \cite[Th\'eor\`eme 2.5]{cl2} to be $-\frac{2\deg(\sM^{g})}{g}\widehat{h}_{\sM}(q)$. This implies
    \[
        \frac{\adeg((\ov{\sM}+\ov{\sQ})^{g+1})}{(g+1)\deg((\sM+\sQ)^{g})}= - \widehat{h}(q) = -\widehat{h}(\sQ).
    \]
\end{proof}

This can be applied to compute the height of a semiabelian variety.

\heightsemiabelian*

\begin{proof}
    By applying the results of Section \ref{prelim_okounkov} and Lemma \ref{semiabelian_okounkov_body}, we see that
    \begin{align*}
        &h_{\rho(D)^{\can}\otimes \pi^*\bar{\sM}}(\ov{G})\\ =& (d+1)!\int_{\Delta(\pi^*\sM+\rho(D))} G_{\ov{G}_\bullet,\pi^*\ov{\sM}+\rho(D)^{\can}}(x) dx
        \\=& -(d+1)!\int_{\Delta\times \Delta_{A_\bullet}(\sM)} \widehat{h}(c(m))d(m,x)\\
        =& -\frac{(d+1)!\deg(\sM)}{g!}\int_{\Delta\times \Delta_{A_\bullet}(\sM)} \widehat{h}(c(m))dm\\
        =& -\frac{(d+1)!\deg(\sM)}{g!\vol(\sM)}\int_{\Delta} \widehat{h}(c(m))dm\\
        =& -(d+1)!\int_{\Delta} \widehat{h}(c(m))dm.
    \end{align*}
\end{proof}

We proceed to study the successive minima of $\ov{G}$. Let $\sF(\Delta)^{i}$ denote the set of $i$-dimensional faces of $\Delta$.

\minimasemiabelian*

\begin{proof}
    By Section \ref{sec:suc_min} it suffices to study successive minima on each locally closed subvariety of $\ov{G}$ corresponding to a torus orbit. Since the functions of the form $h_\Delta$ are invariant under the action of torsion points on $A$ each subvariety each such filtration of $A$ has only one step. In particular, the first step in the height filtration maps surjectively onto $A$ and hence $\zeta_i(\bar{G}) = \zeta_{1}(\bar{G})$ for $i\leq g+1$.

    We start out by noting that the formula holds for $\zeta_{\ess}$ by applying Lemma \ref{semiabelian_okounkov_body} and Theorem \ref{thm:ballay}. By this discussion it suffices to show that the essential minimum of the restriction of $\pi^* \ov{\sM} + \rho(D)^{\can}$ to the closed toric subbundle $\calV(\sigma)$ given by the cone $\sigma$ is $-\min_{m \in F_\sigma} \widehat{h}(c(m))$, where $F_\sigma$ is the face dual to $\sigma$. For this it suffices to understand the Boucksom-Chen transform on each closed toric subbundle.

    Let us understand the restriction of the $\T$-Cartier divisor $D$ on $\toricvar$ to the closure of a $\T$-orbit $V(\sigma)$. By \cite[Proposition 3.4.11]{toricheights}, its restriction is given by the divisor on $V(\sigma)$ corresponding to the face $F_\sigma$ dual to $\sigma$. The restriction of $\rho(D)$ to the toric subbundle is given precisely by $\rho(F_\sigma)$. Hence, the Okounkov body and the Boucksom-Chen transform of the restriction are precisely the restriction to the preimage of $F_\sigma$ and the claim follows. Suppose otherwise that $m_\sigma\in F_\sigma$. Then, the Cartier divisor associated to $F_\sigma-m_\sigma$ intersects $V(\sigma)$ properly. Hence, the restriction of $\rho(D)$ to $\calV(\sigma)$ is given by $\rho(F_\sigma-m_\sigma) + \pi^* c(m_\sigma)$. Hence, the Okounkov body and the Boucksom-Chen transform of the restriction are (up to translation) the restriction to the preimage of $F_\sigma$ and the claim follows.
\end{proof}

We may now apply the above for the specific compactification studied in \cite{cl2}. Let us recall their setup.

Let $G$ be a semiabelian variety with torus $\bbG_m^t$ and associated isomorphisms $N \cong \bbZ^t$ and $M \cong \bbZ^t$. Let $A$ denote the abelian quotient of $G$ and $\sM$ an ample symmetric line bundle on $A$ inducing a polarization morphism $\phi:A \to A^\vee$. Let $e_1,\dots,e_t$ denote the standard basis of $M$. Let $q_1,\dots,q_t \in A(\Kbar)$ satisfying $c(e_i) = \phi(q_i)$. The compactification is taken to be $\bbG^t_m \subset \bbP^n$. If $\Delta^t$ denotes the unit simplex, we let $\Delta = (t+1) \Delta^t - \sum_i e_i$. The line bundle on $\ov{G}$ that Chambert-Loir denotes as $\sL$ is $\rho(\Delta)$.

\begin{lemma}[Lemma 4.5 \cite{cl2}]
    Let $q = \sum_i q_i$. Then, one has
    \[
        \zeta_{\abs}(\rho(\Delta)^{\can} + \pi^* \ov{\sM}) = -\max\{\widehat{h}(q),\underset{i}{\max}\widehat{h}(q - (t+1)q_i)\}.
    \]
\end{lemma}

\begin{proof}
    By Theorem \ref{thm:successive_minima_formula_semiabelian}, the absolute minimum is given by the negative of $\max \widehat{h}(c(m))$ as $m$ ranges over the vertices of $\Delta$. On the vertices $v_0,\dots,v_t$ of $\Delta$, one has $c(v_0) = \phi(-q)$ and $c(v_i) = \phi((t+1)q_i-q)$ for $i=1,\dots,t$.
\end{proof}

\begin{lemma}
    The height of $\ov{G}$ can be computed as
    \[
        h_{\rho(D)^{\can}\otimes \pi^*\bar{\sM}}(\ov{G}) = - \frac{(d+1)\deg(\sM)}{(t+1)(t+2)} \left( \widehat{h}(q) + \sum^t_{i=1} \widehat{h}(q-(t+1)q_i)\right).
    \]
\end{lemma}

\begin{proof}
    By Theorem \ref{thm:height_formula_semiabelian} we are reduced to computing the integral $\int_\Delta \widehat{h}(c(m))$. 
    
    Let $H$ be a symmetric multilinear form in $r$ entries on $\bbR^t$ and $\Delta$ a polytope spanned by $v_0,\dots,v_t$. Then, it is proven in \cite{lasserre_avrachenkov_polytope_integration} that
    \[
        \int_\Delta H(m,\dots,m) = \frac{\vol(\Delta)}{{t+r\choose r}} \sum_{0\leq i_1\leq\dots\leq i_r\leq t} H(v_{i_1},\dots,v_{i_r}).
    \]
    Let us specialize to the case that $r = 2$ and $\sum^t_{i=0} v_i = 0$. Then we obtain
    \[
        \int_\Delta H(m,\dots,m) = \frac{\vol(\Delta)}{(t+1)(t+2)} \sum^t_{i=0} H(v_i,v_i)
    \]
    by subtracting $\frac{\vol(\Delta)}{2{t+r\choose r}}H(v_0+\dots+v_t,v_0+\dots+v_t) = 0$. We apply this to $\Delta$ as in the setup and $H$ the polarization of $\widehat{h}(c(m))$. We lastly need the geometric identity from Theorem \ref{thm:arithmetic_volume}.
\end{proof}

\medskip
\sloppy
\emergencystretch=1em
\printbibliography

\end{document}